\documentclass[10pt]{scrartcl}
\usepackage[T1]{fontenc}
\usepackage[utf8]{inputenc}
\usepackage{amsfonts}
\usepackage{mathrsfs}
\usepackage{bbm}
\usepackage{amsfonts}
\usepackage{amssymb,amsmath,graphicx}
\usepackage{amsthm}
\usepackage{float} \usepackage[colorlinks=true]{hyperref}
\hypersetup{urlcolor=blue, citecolor=red}

\newcommand{\e}{\mathrm{e}}
\newcommand{\Ombar}{\overline{\Omega}}
\newcommand{\io}{\int_\Omega}
\newcommand{\nep}{n_{\varepsilon}}
\newcommand{\cep}{c_{\varepsilon}}
\newcommand{\uep}{\mathbf{u}_{\varepsilon}}
\newcommand{\f}[2]{\frac{#1}{#2}}
\newcommand{\norm}[2][]{\|#2\|_{#1}}
\newcommand{\Lom}[1]{L^{#1}(\Omega)}
\newcommand{\projection}{\mathcal{P}}
\newcommand{\eps}{\varepsilon}

\usepackage{newunicodechar}
\newunicodechar{ℝ}{\mathbb{R}}
\newunicodechar{ℕ}{\mathbb{N}}
\newunicodechar{φ}{\varphi}
\newunicodechar{μ}{\mu}
\newunicodechar{∞}{\infty}
\newunicodechar{γ}{\gamma}
\newunicodechar{∇}{\nabla}

\allowdisplaybreaks

\usepackage{xcolor}

\title{Generalized solutions to a chemotaxis-Navier-Stokes system with arbitrary superlinear degradation}
\author{Mengyao Ding$^1$,~~Johannes Lankeit$^2$\\[6pt]
\footnotesize$^{\textrm 1}$ School of Mathematical Sciences, Peking University,
100871 Beijing, PR China\\
\footnotesize$^{\textrm{ 2}}$
Leibniz Universität Hannover, Institut für Angewandte Mathematik, Welfengarten 1, 30167 Hannover, Germany
}
\date{}

\newtheorem{theorem}{Theorem}[section]
\newtheorem{definition}[theorem]{Definition}
\newtheorem{lemma}[theorem]{Lemma}

\newtheorem{corollary}[theorem]{Corollary}
\newtheorem{remark}[theorem]{Remark}

\catcode`@=11 \@addtoreset{equation}{section} \catcode`@=12

\begin{document}\maketitle
\begin{abstract}
\noindent\textbf{Abstract.} In this work, we study a chemotaxis-Navier-Stokes model in a two-dimensional setting as below,
  \begin{eqnarray}
 \left\{
    \begin{array}{llll}
       \displaystyle n_{t}+\mathbf{u}\cdot\nabla n=\Delta n-\nabla \cdot(n\nabla c)+f(n), &&x\in\Omega,\,t>0,\\
        \displaystyle  c_{t}+\mathbf{u}\cdot\nabla c=\Delta c - c+ n, &&x\in\Omega,\,t>0,\\
        \displaystyle \mathbf{u}_{t}+\kappa(\mathbf{u}\cdot\nabla)\mathbf{u}=\Delta \mathbf{u} +\nabla P+ n\nabla\phi, &&x\in\Omega,\,t>0,\\
        \displaystyle \nabla\cdot\mathbf{u}=0,&&x\in\Omega,\,t>0.\\
    \end{array}
 \right.
\end{eqnarray}
Motivated by a recent work due to Winkler,
we aim at investigating generalized solvability for the model the without
imposing a critical superlinear exponent restriction on the logistic source function $f$. Specifically, it is proven in the present work that there exists a triple of integrable functions $(n,c,\mathbf{u})$ solving the system globally in a generalized sense provided that $f\in C^1([0,\infty))$
satisfies $f(0)\ge0$ and $f(n)\le rn-\mu n^{\gamma}$ ($n\ge0$) with any $\gamma>1$.
Our result indicates that persistent Dirac-type singularities can be ruled out
in our model under the aforementioned mild assumption on $f$.
After giving the existence result for the system, we also show that the generalized solution
exhibits eventual smoothness as long as $\mu/r$ is sufficiently large.
\\[10pt]
\end{abstract}

\section{Introduction}
In this article, we consider the following Keller-Segel-Navier-Stokes system
  \begin{eqnarray} \label{q1}
 \left\{
    \begin{array}{llll}
       \displaystyle n_{t}+\mathbf{u}\cdot\nabla n=\Delta n-\nabla \cdot(n\nabla c)+f(n), &&x\in\Omega,\,t>0,\\
        \displaystyle  c_{t}+\mathbf{u}\cdot\nabla c=\Delta c - c+ n, &&x\in\Omega,\,t>0,\\
        \displaystyle \mathbf{u}_{t}+\kappa(\mathbf{u}\cdot\nabla)\mathbf{u}=\Delta \mathbf{u} +\nabla P+ n\nabla\phi, &&x\in\Omega,\,t>0,\\
        \displaystyle \nabla\cdot\mathbf{u}=0,&&x\in\Omega,\,t>0,\\
        \displaystyle \partial_{\nu}n=\partial_{\nu}c=\mathbf{u}=0,&&x\in\partial\Omega,\,t>0,\\
        \displaystyle n(x,0)=n_0(x),c(x,0)=c_0(x),\mathbf{u}(x,0)=\mathbf{u}_0(x),&&x\in\Omega,\\
    \end{array}
 \right.
\end{eqnarray}
which models a chemotactically active species in a fluid environment, in a bounded domain $\Omega\subsetℝ^2$. While global classical solvability cannot be expected for arbitrary initial data due to the propensity for blow-up inherent in the chemotaxis subsystem, solvability in a generalized sense has recently been proven in the fluid-free setting, \cite{W2019}. In the present work we show that a framework of generalized solvability can even cope with the coupling to a Navier--Stokes fluid ($\kappa=1$), despite the latter negatively affecting a priori known regularity properties of the system.\\

\textbf{The system.} The quantities in \eqref{q1} denote a density $n$ of cells (for example, of bacteria); the concentration $c$ of a chemical signal towards higher concentrations of which the cells move, as indicated by the cross-diffusive ``chemotaxis term''; and the velocity field $\mathbf{u}$ and pressure $P$ of the fluid, which transports both cells and chemical, and in turn is affected by buoyancy forces driven by differences in density between cells and fluid. The given function $\phi\in W^{2,\infty}(\Omega)$ models the gravitational potential causing the buoyancy, and $f\in C^1([0,\infty))$ is used to describe the growth of the population and will be a generalization of the typical choices of logistic-type reproduction. 
Without fluid motions included,
  the model (\ref{q1}) turns
 to be the system
\begin{eqnarray}\label{q2} \left\{
\begin{array}{llll}
 \displaystyle  n_t=\Delta n-\nabla\cdot(n\nabla c)+f(n), & x\in\Omega,~~t>0,\\[4pt]
 \displaystyle  c_t=\Delta c-c+n,& x\in\Omega,~~t>0,
\end{array}\right. \end{eqnarray}
considered in bounded domains $\Omega\subsetℝ^d$, which has attracted great attention for decades (cf. e.g. the surveys \cite{Horstmann,BBTW,LW2019}).
When $f(s)\equiv0$, the system -- then the classical Keller--Segel system -- admits distinct solution behaviour, from global existence to blow-up of solutions, depending on the spatial dimension and the initial data and their mass (cf. the above-mentioned surveys and references therein, in particular \cite{OY,NSY,HV,MW,Wi2013,Wi2010}).\\
To what extent exactly the presence of logistic-type terms $f(s)=rs-\mu s^{\gamma}$ with $\gamma>1$, $\mu>0$, $r\inℝ$, in the first equation of \eqref{q1} or \eqref{q2} hinders the occurrence of blow-up is still subject of ongoing research. Several partial results, however, are known:\\
In the most prototypical case of $\gamma=2$, solutions are global and classical if $d\le 2$, \cite{OTYM}, and if $d\ge3$ according to \cite{W2010} there is $\mu_0>0$ such that $\mu>\mu_0$ ensures global classical solvability. The dependence of an upper bound for $\mu_0$ on other parameters in the system has been investigated by Xiang in \cite{X} following the approach of \cite{W2010}.\\
Without imposing any largeness restriction
on $\mu$, the solvability in the weak sense was determined in \cite{L} for any $d\ge3$ and $r>0$, the
eventual smoothness of these weak solutions was also discussed under the condition that $d=3$ and $r$ is sufficiently small in \cite{L}.\\
In parabolic--elliptic relatives of the system -- for which, in principle, boundedness results similar to those reported for \eqref{q2} are available (see \cite{TelloWin}) -- blow-up has been detected in some cases: First in \cite{W2011} for $d\ge 5$ and
$0<\gamma<\frac{3}{2}+\frac{1}{2d-2}$ in a system where the second equation of \eqref{q2} is replaced by $0=\Delta c-\frac{1}{|\Omega|}\int_{\Omega} n+n$, but also for systems with $0=\Delta c - c + n$ and in $d\ge3$ such results have more recently been obtained, \cite{WZAMP,F2020,BFL}.\\
Recently, considering a suitably designed framework of generalized solvability,
Winkler \cite{W2019} showed that global solution $(n, c)$ with $n\in L_{loc}^{1}\big(\overline{\Omega} \times[0, \infty)\big)$ can be constructed
under the mere hypothesis that
$f$ satisfies $f(0)\ge0$ and
 \begin{align}\label{i}
\frac{f(s)}{s}\rightarrow-\infty \quad \text{as}\quad s\rightarrow\infty.
\end{align}
This result indicates that the mild condition \eqref{i} can rule out the occurrence of persistent Dirac-type singularities in the model \eqref{q2}.
Inspired by \cite{W2019}, the present work is devoted to proving that there is no critical superlinear exponent on the logistic function $f$ for ensuring the generalized solvability when the model is coupled by the Navier-Stokes equation.\\
The interest in the coupling of fluid equations to chemotaxis systems, although initially focussed on systems where the signal substance is consumed (see the overview in \cite[Sec. 4.1.1]{BBTW} or the derivation in \cite{bellomo_bellouquid_chouhad}) also extends to settings with signal production, see e.g. \cite{kiselev_ryzhik,espejo_suzuki}, in the context of broadcast spawning of corals, or \cite{B2018,wu_xiang,yu_wang_zheng,LW,WX2015,Zheng2016}. If the fluid flow is governed by the full Navier--Stokes equations, the resulting system is \eqref{q1} with $\kappa=1$.
Classical solvability necessarily can only be expected in 
settings where both the Keller--Segel subsystem and the Navier--Stokes equations have classical solutions.
For a small-data existence result in the case of $f\equiv 0$ see \cite{yu_wang_zheng}.
Several further findings concerning classical or weak solutions rest on stronger nonlinear diffusion, \cite{B2018,LW,Zheng2016}, a decaying sensitivity function \cite{WX2015,LW,Zheng2016} or, especially for $d\ge3$, simplification of the fluid flow to the Stokes equation \cite{WX2015,Zheng2016}.\\
When the logistic term $f(s)=rs-\mu s^2$ is involved,
Tao and Winkler \cite{TW2016} proved that the model \eqref{q1} with $d=2$ admits a global and bounded classical solution regardless of the size of $r\ge0,\mu>0$. For $d\ge 3$ with a Stokes-governed flow ($\kappa=0$), classical solutions were found in \cite{TW2015} if $μ>23$.
These results, however, rely on $\gamma=2$. For less than quadratic degradation terms in $f$,
generalized solutions have recently been found (see \cite{mengyao} for $d=2$ and \cite{WWX21} for $d=3$),
but only for the case of Stokes-fluid.\\
Motivated by \cite{W2019}, the purpose of the present paper is to study the global solvability of the model \eqref{q1} under a mild
assumption on the logistic function.\\
For the initial data, we assume
\begin{equation}\label{initial}
\left\{\begin{array}{l}{n_{0} \in L^{1}(\Omega) \text { with } n_{0} \geq 0 \text { in } \Omega \text { and } n_{0} \not \equiv 0}, \\
c_{0} \in D((-\Delta+1)^{2\sigma}) \quad \text{for some }\sigma\in(0,\frac14),\\
{\mathbf{u}_{0}\in L^{2}\left(\Omega ; \mathbb{R}^{2}\right)
 ~\text {and } \nabla \cdot \mathbf{u}_{0}=0\text { in}~\mathcal{D}'(\Omega). }
\end{array}\right.
\end{equation}
where $\Delta$ is the Neumann-Laplacian on $\Lom2$ (see also Sec.~\ref{operator}). 
With these, our main result
reads as follows.

\begin{theorem} \label{th 1}
Let $\Omega\subset \mathbb{R}^{2}$ be a bounded domain with smooth boundary. Suppose that $f\in C^{1}\big([0,\infty)\big)$ is such that
\begin{align}\label{f}
f(0)=0,\quad f(s)\le rs-\mu s^{\gamma} ~~ for~all~s\ge0
\end{align}
with $r\in ℝ$, $\mu>0$ and $\gamma>1$. Then for any given initial data $(n_0,c_0,\mathbf{u}_0)$ satisfying \eqref{initial},
there exist functions
\begin{equation}\label{solnreg}
\left\{\begin{array}{l}{n \in L_{loc}^{1}(\overline{\Omega} \times[0, \infty)),} \\
{c\in L_{l o c}^{2}\big([0, \infty) ; W^{1,2}(\Omega)\big),}\\
{\mathbf{u}\in  L_{l o c}^{2}\big([0, \infty) ; W_0^{1,2}(\Omega;ℝ^2)\big)}\end{array}\right.
\end{equation}
with the property that $(n,c,\mathbf{u})$ forms a global generalized solution of \eqref{q1} in the sense of Definition \ref{def3}
below.
\end{theorem}
%

Whereas Theorem~\ref{th 1} answers the question of existence rather completely, the regularity of solutions as guaranteed by \eqref{solnreg} is far from that of classical solutions. On the other hand, for $\gamma=2$ the system (at least without fluid) is known to regularize its solutions: It has been shown (\cite{lankeit_2020}) that even initial data in $L^1\times W^{1,2}$ result in solutions that are smooth in $\Omega\times(0,\infty)$ in two-dimensional domains $\Omega$, and in a three-dimensional setting solutions become smooth after some waiting time if $r$ is not too large, \cite{L}.
Also in related chemotaxis-fluid systems with logistic sources it has been demonstrated that for large times regularity or even convergence can be achieved, see e.g. \cite{L2016,yulan} for a related system with signal consumption and \cite{win_JFA} for a stabilization result concerning solutions of \eqref{q1} with $γ=2$ and small $r$.
Here we show that the same eventual smoothness already occurs in systems with much weaker degradation ($γ>1$), again under a smallness condition on $r$ (cf. \cite{win_JFA,L2016}):

\begin{theorem} \label{th 2}
Let $\Omega\subset \mathbb{R}^{2}$ be a bounded domain with smooth boundary. Let $\gamma>1$. Then there is $\mu_0=\mu_0(\gamma,\Omega)>0$ such that for every function $f$ fulfilling \eqref{f} with $r\in \mathbb{R}$ and $μ>μ_0r_+$ and for all 
initial data $(n_0,c_0,\mathbf{u}_0)$ as in \eqref{initial}, the global generalized solution $(n,c,\mathbf{u})$ of \eqref{q1} constructed in Theorem \ref{th 1} satisfies
\begin{align}
n\in C^{2,1}(\overline{\Omega}\times[T,\infty)), ~~c\in C^{2,1}(\overline{\Omega}\times[T,\infty)),~~
\mathbf{u}\in C^{2,1}(\overline{\Omega}\times[T,\infty))
\end{align}
with some $T>0$.
\end{theorem}
\begin{remark}
For the fluid-free system it has recently been shown that for $γ>1$ (more generally, for $γ\ge 2-\f2d$) and sufficiently large $\f{μ^2}{r^{3-γ}}$, solutions $(n,c)$ essentially converge to the spatially homogeneous equilibrium in $\Lom1\times\Lom2$ as $t\to\infty$, \cite{win_ANS}. The study relies on a Lyapunov functional in which, when evaluated for solutions of \eqref{q1}, all fluid-terms vanish immediately. It is therefore most likely to be expected that solutions of \eqref{q1} have the same property.
\end{remark}

\section{The solution concept}
While the notion of solution for the first equation in \eqref{q1} is somewhat more involved, the second and third equation can be understood in a rather usual weak sense.
In order to indicate the space to which $\mathbf{u}$ belongs, we introduce the solenoidal subspace of $L^2(\Omega;ℝ^2)$ as
\[
L_{\sigma}^{2}(\Omega):=\left\{\varphi \in L^{2}(\Omega;\mathbb{R}^2) \mid \nabla \cdot \varphi=0\text { in}~\mathcal{D}'(\Omega)\right\}
\]
and abbreviate $W_{0,\sigma}^{1,2}(\Omega;ℝ^2)=W_0^{1,2}(\Omega;ℝ^2)\cap L^2_{\sigma}(\Omega)$.

\begin{definition} \label{def2}
A pair $(\mathbf{u},n)$ of functions
\begin{equation}
\left\{\begin{array}{l}{ \mathbf{u }\in L_{l o c}^{2}\big([0,\infty) ; W_{0,\sigma}^{1,2}(\Omega;\mathbb{R}^2)\big)\quad} \\
{n\in L_{loc}^{1}\big(\overline{\Omega} \times[0, \infty)\big)}\end{array}\right.
\end{equation}
satisfying $n\ge0$ is said to globally solve the equation
  \begin{eqnarray*}
 \left\{
    \begin{array}{llll}
        \displaystyle \mathbf{u}_{t}+\kappa(\mathbf{u}\cdot\nabla)\mathbf{u}=\Delta \mathbf{u} +\nabla P+ n\nabla\phi, &&x\in\Omega,\,t>0,\\
        \displaystyle \nabla\cdot\mathbf{u}=0, &&x\in\Omega,\,t>0,\\
        \displaystyle \mathbf{u}=0, &&x\in\partial\Omega,\,t>0,\\
        \displaystyle \mathbf{u}(0)=\mathbf{u}_0,&&x\in\Omega\\
    \end{array}
 \right.
\end{eqnarray*}
in the weak sense if
\begin{equation}\label{def22}
-\int_{\Omega}\mathbf{u}_0\varphi(0)
-\int_0^\infty\int_{\Omega}\mathbf{u}\varphi_t
-\kappa \int_{0}^{\infty} \int_{\Omega} \mathbf{u}\otimes\mathbf{u} \cdot \nabla \varphi
=\int_0^\infty\int_{\Omega}n\nabla\phi\cdot\varphi -\int_0^\infty\int_{\Omega}\nabla \mathbf{u}\cdot\nabla\varphi
\end{equation}
holds for every $\varphi\in C_{0}^{\infty}\left((\Omega; \mathbb{R}^{2}) \times[0, \infty)\right)$ with $\nabla \cdot \varphi \equiv 0$ in $\Omega \times(0, \infty)$.
\end{definition}

\begin{definition} \label{def1}
A triple $(n,c,\mathbf{u})$ of functions
\begin{equation}\label{def11}
\left\{\begin{array}{l}{n \in L_{loc}^{1}(\overline{\Omega} \times[0, \infty)),} \\
{c\in L_{l o c}^{2}\big([0, \infty) ; W^{1,2}(\Omega)\big),}\\
{\mathbf{u}\in  L_{l o c}^{2}\big([0, \infty) ; W_{0,\sigma}^{1,2}(\Omega;ℝ^2)\big)}\end{array}\right.
\end{equation}
satisfying $n\ge0$, $c\ge0$ is said to globally solve the problem
  \begin{eqnarray*}
 \left\{
    \begin{array}{llll}
        \displaystyle  c_{t}+\mathbf{u}\cdot\nabla c=\Delta c - c+ n, &&x\in\Omega,\,t>0,\\
        \displaystyle \partial_{\nu}c=0, &&x\in\partial\Omega,\,t>0,\\
        \displaystyle c(0)=c_0,&&x\in\Omega\\
    \end{array}
 \right.
\end{eqnarray*}
in the weak sense if 
\begin{align}\label{def12}
-\int_{\Omega}c_0\varphi(0)-\int_0^\infty\int_{\Omega}c\varphi_t &=-\int_0^\infty\int_{\Omega}\nabla c\cdot\nabla\varphi +\int_0^\infty\int_{\Omega}c\mathbf{u}\cdot\nabla \varphi
+\int_0^\infty\int_{\Omega}n\varphi
-\int_0^\infty\int_{\Omega}c\varphi
\end{align}
holds for every $\varphi\in
C^{\infty}_0(\overline{\Omega}\times[0,\infty))$.
\end{definition}
With the two latter components weakly solving the corresponding equations, the generalized solution  $(n,c,\mathbf{u})$ of \eqref{q1}
is exhibited as $n$ satisfies the first equation of \eqref{q1} in a certain form which was first established in \cite{W2015s}. It combines a logarithmic supersolution property of $n$ (cf. also \cite{LL}) with an upper estimate for its mass.
\begin{definition} \label{def3}
Let a triple $(n,c,\mathbf{u})$ of functions
\begin{equation}
\left\{\begin{array}{l}{n \in L_{loc}^{1}(\overline{\Omega} \times[0, \infty)) ,} \\
{c\in L_{l o c}^{1}\left([0, \infty) ; W^{1,2}(\Omega)\right),} \\
{\mathbf{u}\in L_{l o c}^{1}\big([0,\infty) ; W_0^{1,2}(\Omega;\mathbb{R}^2)\big)}\end{array}\right.
\end{equation}
satisfy $n\ge0$, $c\ge0$ and be such that
\begin{align}
f(n)\in L^1_{loc}(\overline{\Omega}\times[0,\infty)).
\end{align}
Then $(n,c,\mathbf{u})$ will be called a global generalized solution of \eqref{q1} if \eqref{def11}-\eqref{def22} are satisfied,
and if
\begin{align}
\int_{\Omega} n(\cdot, t) \leq \int_{\Omega} n_{0}+\int_{0}^{t} \int_{\Omega} f(n) \quad \text { for }~ a.e.~t>0,
\end{align}
and if
\begin{align}
(n+1)^{-2}|\nabla n|^2\in L^1_{loc}\big(\overline{\Omega}\times[0,\infty)\big),
\end{align}
and if
\begin{align}\label{def32}
&\quad\int_{0}^{\infty} \int_{\Omega}\ln(n+1)\varphi_{t}
+
\int_{\Omega} \ln(n_0+1)\varphi(0)\nonumber\\
 &\le-\int_{0}^{\infty}\int_{\Omega}\ln(n+1) \mathbf{u}\cdot\nabla \varphi
 -\int_{0}^{\infty}\int_{\Omega} |\nabla\ln(n+1)|^2\varphi
 \nonumber\\
 &\quad+\int_{0}^{\infty}\int_{0}^{\infty}\int_{\Omega} (n+1)^{-1}\nabla n\cdot\nabla\varphi
 +\int_{0}^{\infty}\int_{\Omega} (n+1)^{-2} n
 \nabla n\cdot\nabla c\varphi
 \nonumber\\
  &\quad-\int_{0}^{\infty}\int_{\Omega} (n+1)^{-1} n
  \nabla c\cdot\nabla\varphi
 -\int_{0}^{\infty}\int_{\Omega} (n+1)^{-1} f(n)\varphi
\end{align}
holds for each nonnegative function $\varphi\in C^{\infty}_0\big(\overline{\Omega}\times[0,\infty)\big)$.
\end{definition}

Like its famous precursor, the concept of renormalized solutions (see \cite{DiPernaLions}, transferred to the context of chemotaxis systems in \cite{W2015s}), this notion of solvability rests on the idea that it may be easier to derive a priori estimates for or pass to the limit in integrals involving not the solution itself, but a nonlinear transformation thereof, here $\ln(n+1)$ instead of $n$.

\newcommand{\A}{\mathcal{L}}

For the solutions in \cite{W2019,mengyao,WWX21} (dealing with \eqref{q1} in a fluid-free variant or for a Stokes fluid), a combined quantity of the form $\phi(n)\psi(c)$, with $\phi$ being a bounded, decreasing and convex function, was employed.
This additional step further away from classical solvability was not necessary for Definition~\ref{def3} or Theorem~\ref{th 1}. This is because in the two-dimensional case, we can investigate the regularity of solutions more precisely by applying fractional powers of the operator $\A=-\Delta+1$, which allows us to construct the generalized solution closer to a classical one.
Let us focus on the component $c$ and specifically discuss how the above idea is performed: 
When studying the energy development of $c$, the time-space estimates are established by testing
the equation by $c$ or $\A c$ in the existing literature.
But the resulting $L^\infty L^2$-boundedness of $c$ can not give all desired estimates and
the $L^\infty L^2$-boundedness of $\A^{1/2}c$ seems to require the assumption $\gamma\ge2$.
To handle this difficulty, we turn to the test function $\A^{\beta}c$ with a proper exponent
$\beta$ ensuring that the estimates of $\|\A^{(\beta+1)/2}c\|_{L^\infty L^2}$ can be built on the condition $\gamma>1$ and suffice to proceed to the compactness arguments.
Here, we remark that the aforementioned reasoning needs to be performed under the assumption of two-dimensionality due to the appearance of the convection term $\mathbf{u}\cdot \nabla c$.

\section{Solutions of an approximate system}\label{sec3}
In the following, we fix $\kappa=1$, $\lambda>1$, $\sigma\in(0,\frac14)$ (cf. \eqref{initial}), $r\inℝ$, $\mu>0$, $\gamma>1$, $f$ fulfilling \eqref{f} and let $n_0$, $c_0$, $\mathbf{u}_0$ be as in \eqref{initial}. We furthermore introduce a family of functions
\begin{equation}
\left\{\begin{array}{l}{n_{0\varepsilon} \in C^{1}(\overline{\Omega}) \text { with } n_{0\varepsilon} \geq 0 \text { in }\overline{\Omega} \text { and } n_{0\varepsilon} \not \equiv 0,} \\
{c_{0\varepsilon} \in C^{1}(\overline{\Omega}) \cap D((-\Delta+1)^{2\sigma}) ~ \text { with } c_{0\varepsilon} \geq 0 \text { in }\overline{ \Omega}, }\\
{\mathbf{u}_{0\varepsilon}\in C^{1}(\overline{\Omega};\mathbb{R}^2) ~\text { with } \nabla \cdot \mathbf{u}_{0\varepsilon}=0  ~\text {in }\Omega
 \text { and }  \mathbf{u}_{0\varepsilon} =0 \text { on }\partial\Omega}
\end{array}\right.
\end{equation}
satisfying
\begin{align}\label{initial2}
\int_{\Omega}n_{0\varepsilon} \le 2\int_{\Omega}n_0 ~ \text { for each } \varepsilon\in(0,1).
\end{align}
and
\begin{align}\label{initial1}
n_{0\varepsilon} \rightarrow n_0 ~ \text { in } L^{1}(\Omega),
~c_{0\varepsilon} \rightarrow c_0 ~\text { in } D\left((-\Delta+1)^{2\sigma}\right)~ \text { as well as }
~\mathbf{u}_{0\varepsilon} \rightarrow \mathbf{u}_{0}~ \text { in } L^{2}(\Omega)
\end{align}
as $\varepsilon\rightarrow0$, where $-\Delta$ stands for the Neumann-Laplacian in $L^2(\Omega)$, see \eqref{A1}.\\ 

We then intend to construct the generalized solution whose existence Theorem~\ref{th 1} claims as limit of a sequence of solutions of the approximate system
  \begin{eqnarray} \label{qapp}
 \left\{
    \begin{array}{llll}
       \displaystyle n_{\varepsilon t}+\mathbf{u}_\varepsilon\cdot\nabla n_\varepsilon=\Delta n_\varepsilon-\nabla \cdot(n_\varepsilon\nabla c_\varepsilon)+f(n_\varepsilon)-\varepsilon n_\varepsilon^{2}, &&x\in\Omega,\,t>0,\\
        \displaystyle  c_{\varepsilon t}+\mathbf{u}_\varepsilon\cdot\nabla c_\varepsilon=\Delta c_\varepsilon - c_\varepsilon+ \frac{n_\varepsilon}{1+\varepsilon n_\varepsilon}, &&x\in\Omega,\,t>0,\\
        \displaystyle \mathbf{u}_{\varepsilon t}
        +\kappa(\mathbf{u}_{\varepsilon }\cdot\nabla)\mathbf{u}_{\varepsilon }
        =\Delta \mathbf{u}_\varepsilon +\nabla P_\varepsilon+ n_\varepsilon\nabla\phi, &&x\in\Omega,\,t>0,\\
        \displaystyle \nabla\cdot\mathbf{u}_\varepsilon=0,&&x\in\Omega,\,t>0\\
        \displaystyle \partial_{\nu} n_{\varepsilon}=\partial_{\nu} c_{\varepsilon}=0,  \mathbf{u}_{\varepsilon}=0, & & x \in \partial \Omega, t>0,\\
        \displaystyle n_{\varepsilon}(x, 0)=n_{0\varepsilon}(x),c_{\varepsilon}(x, 0)=c_{0\varepsilon}(x),
         \mathbf{u}_{\varepsilon}(x, 0)=\mathbf{u}_{0\varepsilon}(x), && x \in \Omega,\\
    \end{array}
 \right.
\end{eqnarray}
The modifications in its first two equations (if compared to \eqref{q1}) ensure global classical solvability:
\begin{lemma}\label{lem2.1}
Let $\varepsilon>0$. Then there exist functions $(n_\varepsilon,c_\varepsilon,\mathbf{u}_\varepsilon)$ satisfying
$n_\varepsilon,c_\varepsilon\ge 0$ in $\overline{\Omega} \times[0, \infty)$ and
\begin{equation}
\left\{\begin{array}{l}
{n_{\varepsilon} \in C^{0}(\overline{\Omega} \times[0, \infty)) \cap C^{2,1}(\Ombar\times(0, \infty)), } \\
{c_{\varepsilon} \in C^{0}(\overline{\Omega} \times[0, \infty))\cap C([0,\infty);D((-\Delta+1)^{2\sigma})) \cap C^{2,1}(\Ombar\times(0, \infty)),} \\
{\mathbf{u}_{\varepsilon} \in C^{0}\big((\overline{\Omega};\mathbb{R}^2) \times[0, \infty)\big)
\cap C^{2,1}\big((\Ombar;\mathbb{R}^2)\times(0, \infty)\big),}\end{array}\right.
\end{equation}
together with some $P_\varepsilon\in C^{1,0}(\overline{\Omega} \times(0, \infty))$ such that $(n_\varepsilon,c_\varepsilon,\mathbf{u}_\varepsilon,P_\varepsilon)$ is a classical and global solution of \eqref{qapp}. This solution is unique within the indicated class, up to addition of spatially constant functions to $P_{\eps }$.
\end{lemma}
\begin{proof}
 This is covered by the setting of \cite{TW2016}, with local existence and regularity proven along the lines of \cite[Lemma 2.1]{WinCPDE}.
\end{proof}

Given any $\eps >0$, by $(n_{\eps },c_{\eps },\mathbf{u}_{\eps },P_{\eps })$ we refer to this solution. In a first step we ensure that it satisfies an integral identity resembling that of Definition~\ref{def3}:

\begin{lemma}\label{lem2.2}
Let $\eps >0$. Then for any $\varphi\in C^{\infty}(\overline{\Omega}\times(0,\infty))$, we have
\begin{align}\label{a0}
\quad\int_{\Omega}\partial_{t}\ln\left(n_{\varepsilon}+1\right)\varphi
 &=\int_{\Omega}|\nabla \ln(n_{\varepsilon}+1)|^2\varphi
-\int_{\Omega}\left(n_{\varepsilon}+1\right)^{-1} \nabla n_\varepsilon \cdot\nabla\varphi
 \nonumber\\
 &\quad+\int_{\Omega}\ln\left(n_{\varepsilon}+1\right)\mathbf{u}_\varepsilon\cdot\nabla \varphi
 -\int_{\Omega}\left(n_{\varepsilon}+1\right)^{-2}n_\varepsilon
\nabla n_\varepsilon\cdot\nabla c_\varepsilon \varphi\nonumber\\
  &\quad+\int_{\Omega}\left(n_{\varepsilon}+1\right)^{-1} n_\varepsilon \nabla c_\varepsilon \cdot\nabla\varphi
 +\int_{\Omega}\left(n_{\varepsilon}+1\right)^{-1}
    f(n_\varepsilon)\varphi\nonumber\\
  &\quad-\varepsilon\int_{\Omega}\left(n_{\varepsilon}+1\right)^{-1}
  n^2_{\varepsilon}\varphi
  \qquad\text{ in } (0,\infty).
\end{align}
\end{lemma}
In lieu of proof we give the following more general lemma, from which Lemma~\ref{lem2.2} results upon inserting $F(s)=\ln s$, $F'(s)=\f1s$ and $-F''(s)=\f1{s^2}$.
\begin{lemma}\label{lem:testing}
 Let $F\in C^2((0,\infty))$. Then for every $\eps >0$ and $φ\in C^\infty(\Ombar\times(0,\infty))$,
 \begin{align*}
  \partial_t\io F(\nep+1)φ =& -\io F''(\nep+1)|\nabla\nep|^2 φ - \io F'(\nep +1)\nabla \nep\cdot\nablaφ +\io F(\nep+1)\uep\cdot\nablaφ \\
  &+\io F''(\nep+1)\nep\nabla\nep\cdot\nabla\cepφ +\io F'(\nep+1)\nep\nabla\cep\cdot\nablaφ\\
  &+\io F'(\nep+1)f(\nep)φ - \eps \io \nep^2F'(\nep+1)φ \qquad \text{in } (0,\infty).
 \end{align*}
\end{lemma}
\begin{proof}
Applying the chain rule and inserting $\eqref{qapp}_1$ gives that
\begin{align}\label{a1}
φ\partial_{t}F\left(n_{\varepsilon}+1\right)
=φF'\left(n_{\varepsilon}+1\right)\Big(\Delta n_\varepsilon-\mathbf{u}_\varepsilon\cdot\nabla n_\varepsilon-\nabla \cdot(n_\varepsilon\nabla c_\varepsilon)+f(n_\varepsilon)-\varepsilon n^2_\varepsilon\Big)
\end{align}
in $\Omega\times(0,\infty)$.
Integration by parts shows that
\begin{align}\label{a3}
\int_{\Omega}F'\left(n_{\varepsilon}+1\right) \Delta n_\varepsilon\varphi
&=-\int_{\Omega}F''\left(n_{\varepsilon}+1\right) |\nabla n_{\varepsilon}|^2\varphi-\int_{\Omega}F'\left(n_{\varepsilon}+1\right) \nabla n_\varepsilon \cdot\nabla\varphi
\end{align}
and, since $\nabla\cdot\mathbf{u}_\varepsilon=0$ in $\Omega\times(0,\infty)$,
\begin{align}\label{a2'}
-\int_{\Omega}\left(n_{\varepsilon}+1\right)^{-1} \mathbf{u}_\varepsilon\cdot\nabla n_\varepsilon\varphi
&=-\int_{\Omega}\mathbf{u}_\varepsilon
\cdot\nabla F\left(n_{\varepsilon}+1\right) \varphi
=\int_{\Omega} F\left(n_{\varepsilon}+1\right) \mathbf{u}_\varepsilon\cdot\nabla\varphi,
\end{align}
hold true in $(0,\infty)$, as does
\begin{align}\label{a4}
-\int_{\Omega}F'\left(n_{\varepsilon}+1\right)\nabla \cdot(n_\varepsilon\nabla c_\varepsilon) \varphi
&=-\int_{\Omega}F''\left(n_{\varepsilon}+1\right)n_\varepsilon
\nabla n_\varepsilon\cdot\nabla c_\varepsilon \varphi
+\int_{\Omega}F'\left(n_{\varepsilon}+1\right) n_\varepsilon \nabla c_\varepsilon \cdot\nabla\varphi.
\end{align}
Integrating \eqref{a1} over $\Omega$ and inserting \eqref{a2'}, \eqref{a3} and \eqref{a4}, we conclude the proof.
\end{proof}

\section{Uniform estimates}\label{sec:estimates}

In this first step of the proof we derive the following integrability properties of $n_\varepsilon,f(n_\varepsilon)$.
This result will serve as fundamental ingredient for the forthcoming estimates.

\begin{lemma}\label{lem3.1}
Let $T>0$. Then there exists $C = C(T) > 0$ ensuring
\begin{align}\label{c1}
\int_{0}^{T} \int_{\Omega}|f(n_{\varepsilon})|\leq C
 \quad for~all~\varepsilon\in(0,1).
\end{align}
Moreover, we have
\begin{equation}\label{c2}
\sup_{\varepsilon\in(0,1)}\sup_{t\ge0}\int_{\Omega} n_\varepsilon(\cdot, t) <\infty
\end{equation}
and
\begin{equation}\label{c2'}
\int_{\Omega} n_\varepsilon(\cdot, t) \leq \int_{\Omega} n_{0,\varepsilon}+\int_{0}^{t} \int_{\Omega} f(n_\varepsilon)\quad for~all~t>0~\text{and}~\varepsilon\in(0,1).
\end{equation}
\end{lemma}
\begin{proof}
This results from an integration of \eqref{qapp}$_1$ in exactly the same way as in \cite[Lemma 4.1]{W2019}, since $\nabla\cdot\mathbf{u}_\varepsilon=0$ in $\Omega\times(0, T)$ and $\mathbf{u}_\varepsilon=0$ on $\partial\Omega\times(0, T)$ make the only integral containing the fluid term vanish.
\end{proof}

Based on the structural assumption on $f$, we use an interpolation inequality to state the following result as a direct application of Lemma \ref{lem3.1}.

\begin{lemma}\label{lem3.2}
Let $T>0$. There exists $C = C(T) > 0$ such that
\begin{align}\label{c3}
\| n_{\varepsilon}\|_{L^\gamma((0,T);L^\gamma(\Omega))}\leq C
 \quad for~all~\varepsilon\in(0,1).
\end{align}
Moreover, for any $p,q>1$ satisfying $p\le\gamma$ and $\frac{1}{p}+\frac{\gamma-1}{q}\ge1$, we can find a positive constant $C=C(p,q,T)$ such that
\begin{equation}\label{c4}
\| n_{\varepsilon}\|_{L^q((0,T);L^p(\Omega))}\le C\quad for~all~\varepsilon\in(0,1).
\end{equation}
\end{lemma}
\begin{proof}The claim \eqref{c3} can be directly achieved by combining the assumption \eqref{f} on the form of $f$ with \eqref{c1}.
Let $p,q>1$ be such that $p\le\gamma$ and $\frac{1}{p}+\frac{\gamma-1}{q}\ge1$.
By using the H\"{o}lder inequality, we have
\begin{align}\label{p1}
\int_0^T \Big(\int_{\Omega} n_\varepsilon^p \Big)^\frac{q}{p}
&= \int_0^T \Big(\int_{\Omega} n_\varepsilon^{p(\frac{\gamma-p}{p(\gamma-1)}+1-\frac{\gamma-p}{p(\gamma-1)})} \Big)^\frac{q}{p}\nonumber\\
&\le \int_0^T \Big(\int_{\Omega} n_\varepsilon \Big)^\frac{q(\gamma-p)}{p(\gamma-1)}
 \Big(\int_{\Omega} n_\varepsilon^{\gamma} \Big)^\frac{q(p-1)}{p(\gamma-1)}
 \nonumber\\
&\le\Big(\sup_{t\in(0,T)}\|n_\varepsilon(\cdot,t)\|_{L^1(\Omega)}\Big)^\frac{q(\gamma-p)}{p(\gamma-1)}
 \int_0^{T}\Big(\int_{\Omega} n_\varepsilon^{\gamma} \Big)^\frac{q(p-1)}{p(\gamma-1)}
\end{align}
for all $\varepsilon\in(0,1)$.
Since the condition $\frac{1}{p}+\frac{\gamma-1}{q}\ge1$ ensures $\frac{q(p-1)}{p(\gamma-1)}\le 1$, it  holds that
\begin{align}\label{p2}
\Big(\int_{\Omega} n_\varepsilon^{\gamma}(x,t) dx\Big)^\frac{q(p-1)}{p(\gamma-1)}
\le \int_{\Omega} n_\varepsilon^{\gamma}(x,t) dx+1\quad for~all~t>0~\text{and}~\varepsilon\in(0,1).
\end{align}
Thus, we can derive \eqref{c4} from \eqref{c2}, \eqref{c3}, \eqref{p1} and \eqref{p2}.
\end{proof}

Taking $q=2$ and $p=\frac2{3-\gamma}$ in the above lemma, we immediately have the following corollary.
\begin{corollary}\label{cor1}
Let $T>0$. If $\gamma\in(1,2)$, then there exists $C = C(T) > 0$ such that
\begin{equation}\label{c4'}
\| n_{\varepsilon}\|_{L^2((0,T);L^{\frac{2}{3-\gamma}}(\Omega))}\le C\quad for~all~\varepsilon\in(0,1).
\end{equation}
\end{corollary}
We note that the derivation of Lemma~\ref{lem3.2} relies on $f_-$ growing at least as fast as a superlinear power, as opposed to growing superlinearly only. The result in the form of \eqref{c4'} will be a crucial ingredient in the following proofs.
In order to apply this corollary, in several of the upcoming lemmata we will assume $\gamma<2$. They will still allow to conclude Theorem~\ref{th 1} for $\gamma\ge 2$, too (see proof of Theorem~\ref{th 1} at the end of Section~\ref{pfthm1}).
\subsection{Estimates for \texorpdfstring{$\{\mathbf{u}_{\varepsilon}\}_{\varepsilon\in(0,1)}$}{\{u\_ε \}}}
The main difference between \eqref{q1} and \eqref{q2} is the additional presence of $\mathbf{u}$.  Fortunately, \eqref{c2}, \eqref{c3}, \eqref{c4'} imply some $\eps $-independent boundedness of $\mathbf{u}_\varepsilon$ and $\nabla\mathbf{u}_\varepsilon$.


\begin{lemma}\label{lem3.3}
Let $T>0$. If $\gamma\in(1,2)$, we can find a positive constant $C=C(T)$ such that
\begin{align}\label{b0}
&\int_{\Omega}|\mathbf{u}_\varepsilon(\cdot, t)|^{2} \leq C\quad \text { for all } t\in(0,T)\text { and } \varepsilon \in(0,1)
\end{align}
and
\begin{align}\label{b0'}
&\int_{0}^{T}\int_{\Omega}|\nabla\mathbf{u}_\varepsilon(x, t)|^{2}dxdt \leq C\quad \text { for all } \varepsilon \in(0,1).
\end{align}
\end{lemma}
\begin{proof}
By using the imbedding $W^{1,2}_0(\Omega) \hookrightarrow L^{\frac{2}{\gamma-1}}(\Omega)$ and the Poincar\'{e} inequality,
we can find $c_{1}=c_1(\gamma,\Omega)>0$ such that
$$
\|\mathbf{u}_\varepsilon\|_{L^{\frac{2}{\gamma-1}}(\Omega)}
 \leq c_{1}\|\nabla \mathbf{u}_\varepsilon\|_{L^{2}(\Omega)}
 ~~\text { for all } \varepsilon \in(0,1).
$$
Testing $\eqref{qapp}_{3}$ by $\mathbf{u}_\varepsilon$, noticing $\nabla \cdot\mathbf{u}_\varepsilon=0$ in $\Omega\times(0,T)$ and defining $c_2:=\|\phi\|_{L^\infty(\Omega)}$,
we use the H\"{o}lder inequality and Young's inequality
to find
$$
\begin{aligned}
\frac{1}{2} \frac{d}{d t} \int_{\Omega}|\mathbf{u}_\varepsilon|^{2}
+\int_{\Omega}|\nabla\mathbf{u}_\varepsilon|^{2}
&=\int_{\Omega} n_\varepsilon \mathbf{u}_\varepsilon\cdot \nabla \phi  \\
& \leq c_2 \|\mathbf{u}_\varepsilon\|_{L^{\frac{2}{\gamma-1}}(\Omega)}
\|n_\varepsilon\|_{L^{\frac{2}{3-\gamma}}(\Omega)} \\
& \leq c_{1}c_2\|\nabla \mathbf{u}_\varepsilon\|_{L^{2}(\Omega)}
\|n_\varepsilon\|_{L^{\frac{2}{3-\gamma}}(\Omega)} \\
& \leq \frac{1}{2} \int_{\Omega}|\nabla\mathbf{u}_\varepsilon|^{2}
+\frac{c_{1}^{2}c_{2}^{2}  }{2}\|n_\varepsilon\|^2_{L^{\frac{2}{3-\gamma}}(\Omega)}~~
\text{ in } (0,T) \text{ for all } \varepsilon \in(0,1),
\end{aligned}
$$
which directly shows that
\begin{equation}\label{j1}
\frac{d}{d t} \int_{\Omega}|\mathbf{u}_\varepsilon|^{2}
+\int_{\Omega}|\nabla \mathbf{u}_\varepsilon|^{2}
\leq
c^2_1c_{2}^{2}\|n_\varepsilon\|^2_{L^{\frac{2}{3-\gamma}}(\Omega)}
~~\text{ in } (0,T) \text{ for all } \varepsilon \in(0,1).
\end{equation}
 An integration of \eqref{j1} over $(0,\tau)$ with
any $\tau\in(0,T]$ gives that
\begin{align}\label{j2-1}
\int_{\Omega}|\mathbf{u}_\varepsilon(\cdot,\tau)|^{2}
+\int_{0}^{\tau}\int_{\Omega}|\nabla \mathbf{u}_\varepsilon|^{2}
\leq
&c^2_1c_{2}^{2}\int_{0}^{\tau}\|n_\varepsilon(\cdot,t)\|^2_{L^{\frac{2}{3-\gamma}}(\Omega)}dt
+\int_{\Omega}|\mathbf{u}_{0\varepsilon}|^{2}\nonumber\\
\leq&
c^2_1c_{2}^{2}\int_{0}^{T}\|n_\varepsilon(\cdot,t)\|^2_{L^{\frac{2}{3-\gamma}}(\Omega)} dt+\int_{\Omega}|\mathbf{u}_{0\varepsilon}|^{2}
\quad \text { for all }\varepsilon \in(0,1).
\end{align}
Thus, by utilizing \eqref{c4'} and \eqref{initial1}, we can deduce the existence of
$c_3=c_3(\gamma,T,\Omega)>0$ such that
\begin{align}\label{j2-2}
\sup_{\tau\in(0,T)}\int_{\Omega}|\mathbf{u}_\varepsilon(\cdot,\tau)|^{2}
+\int_{0}^{T}\int_{\Omega}|\nabla \mathbf{u}_\varepsilon|^{2}
\leq
&c_3
\quad \text { for all }\varepsilon \in(0,1),
\end{align}
as desired.
\end{proof}

\begin{corollary}\label{cor2}
If $\gamma\in(1,2)$, letting $T>0$, we can find $C = C(T) > 0$ satisfying
\begin{align}\label{w0}
\int_{0}^{T} \int_{\Omega}|\mathbf{u}_\varepsilon|^4\leq C
 \quad for~all~\varepsilon\in(0,1).
\end{align}
\end{corollary}
\begin{proof}
An application of the Gagliardo-Nirenberg inequality allows us to find $c_1>0$ satisfying
\begin{align}\label{w1:1}
\|\mathbf{u}_\varepsilon\|_{L^4(\Omega)}
&\le  c_1\|\nabla\mathbf{u}_\varepsilon(\cdot,t)\|^\frac{1}{2}_{L^2(\Omega)}
\|\mathbf{u}_\varepsilon(\cdot,t)\|^\frac{1}{2}_{L^2(\Omega)} \text{  for all } \varepsilon\in(0,1),
\end{align}
which together with Lemma \ref{lem3.3} implies that for some $c_2=c_2(T)>0$
\begin{align}\label{w1:2}
\int_0^T \int_{\Omega}|\mathbf{u}_\varepsilon(x,t)|^4dxdt
&\le c_1^4 \int_0^T \|\nabla\mathbf{u}_\varepsilon(\cdot,t)\|^2_{L^2(\Omega)}
\|\mathbf{u}_\varepsilon(\cdot,t)\|^{2}_{L^{2}(\Omega)}dt
 \nonumber\\
&\le c_1^4\Big(\sup_{t\in(0,T)}\|\mathbf{u}_\varepsilon(\cdot,t)\|_{L^{2}(\Omega)}\Big)^{2}
 \int_0^{T}\|\nabla\mathbf{u}_\varepsilon(\cdot,t)\|^2_{L^2(\Omega)}dt
 \le c_2
\end{align}
for all $\varepsilon\in(0,1)$.
\end{proof}
\begin{lemma}\label{lemL2}
Let $T>0$. If $\gamma\in(1,2)$, then
$\{\mathbf{u}_\varepsilon\}_{\varepsilon\in(0,1)}$ is relatively compact with respect to the strong topology in $L^2\big((0, T);L^{2}(\Omega;\mathbb{R}^2)\big)$.
\end{lemma}
\begin{proof}
We multiply \eqref{qapp}$_3$ by $\xi \in C_0^\infty(\Omega;\mathbb{R}^2)$ with $\nabla\cdot\xi=0$
and integrate by parts to get
\begin{align}\label{u6}
\left|\int_{\Omega} \mathbf{u}_{\varepsilon t}(\cdot, t)\xi\right| &=\left|\int_{\Omega}-\nabla\mathbf{u}_\varepsilon\cdot \nabla \xi
-
\kappa\int_{\Omega}(\mathbf{u}_{\varepsilon }\cdot\nabla)\mathbf{u}_{\varepsilon }\xi
+\int_{\Omega} n_{\varepsilon}\nabla\phi\cdot\xi \right|\nonumber\\
 & \leq\left\|\nabla\mathbf{u}_\varepsilon(\cdot, t)\right\|_{L^{2}(\Omega)}\| \nabla \xi\|_{L^{2}(\Omega)}
 +\kappa\left\|\mathbf{u}_\varepsilon(\cdot, t)\right\|_{L^{2}(\Omega)}
 \left\|\nabla\mathbf{u}_\varepsilon(\cdot, t)\right\|_{L^{2}(\Omega)}
 \|\xi\|_{L^{\infty}(\Omega)}\nonumber\\
 &
\quad+\left\|n_{\varepsilon}(\cdot, t)\right\|_{L^{1}(\Omega)}\|\nabla\phi\|_{L^{\infty}(\Omega)}
\|\xi\|_{L^{\infty}(\Omega)}
~~\text{for~all~}t\in(0,T)~\text{and}~\varepsilon\in(0,1).
\end{align}
Thus, there exists $c_1>0$ such that
\begin{align*}
\left\|\mathbf{u}_{\varepsilon t}(\cdot, t)\right\|_{\big(W_0^{1,\infty}(\Omega)\cap L^{2}_\sigma(\Omega)\big)^{*}}
&\leq c_1\Big(\left\|\nabla\mathbf{u}_\varepsilon(\cdot, t)\right\|_{L^{2}(\Omega)}
 +\left\|\mathbf{u}_\varepsilon(\cdot, t)\right\|_{L^{2}(\Omega)}
 \left\|\nabla\mathbf{u}_\varepsilon(\cdot, t)\right\|_{L^{2}(\Omega)}\nonumber\\
&\quad+\left\|n_{\varepsilon}(\cdot, t)\right\|_{L^{1}(\Omega)}\Big)
\quad\text{for~all~}t\in(0,T)~\text{and}~\varepsilon\in(0,1).
\end{align*}
This along with \eqref{c2'}, \eqref{b0} and \eqref{b0'} indicates that
\begin{align}\label{u7}
\left\{\mathbf{u}_{\varepsilon t}\right\}_{\varepsilon \in(0,1)} \text { is  bounded in } L^{2}\Big((0, T);\big(W_0^{1,\infty}(\Omega)\cap L^{2}_\sigma(\Omega)\big)^{*}\Big).
\end{align}
 Since $W_0^{1,2}(\Omega)\cap L^{2}_\sigma(\Omega)$ is compactly embedded into $L^2(\Omega)$, and $L^2(\Omega)$ is continuously embedded in $\big(W_0^{1,\infty}(\Omega)\cap L^{2}_\sigma(\Omega)\big)^{*}$,
 we obtain the desired result by \eqref{b0'} and \eqref{u7} due to the Aubin-Lions lemma.\end{proof}

\subsection{Estimates for \texorpdfstring{$\{c_{\varepsilon}\}_{\varepsilon\in(0,1)}$}{\{c\_ε \}}}
\begin{lemma}\label{lem3.5}
Let $T>0$. If $\gamma\in(1,2)$, we can find a positive constant $C=C(T)$ such that
\begin{align}\label{u0}
&\int_{\Omega}c_\varepsilon^2(\cdot, t) \leq C\quad \text { for all } t\in(0,T)\text { and } \varepsilon \in(0,1)
\end{align}
and
\begin{align}\label{u0'}
&\int_{0}^{T}\int_{\Omega}|\nabla c_\varepsilon(x, t)|^{2}dxdt \leq C\quad \text { for all } \varepsilon \in(0,1).
\end{align}
\end{lemma}
\begin{proof}
Testing $\eqref{qapp}_{2}$ by $c_\varepsilon$ and using the H\"{o}lder inequality, we have that
\begin{align}\label{u1}
\frac{1}{2} \frac{d}{d t} \int_{\Omega}c_\varepsilon^{2}
& +\int_{\Omega}|\nabla c_\varepsilon|^{2} +\int_{\Omega}c_\varepsilon^{2}
=\int_{\Omega} \frac{n_\varepsilon c_\varepsilon}{1+\varepsilon n_\varepsilon} \nonumber\\
& \leq\|c_\varepsilon\|_{L^{\frac{2}{\gamma-1}}(\Omega)}
\|n_\varepsilon\|_{L^{\frac{2}{3-\gamma}}(\Omega)} \quad \text { in } (0, T)\text { and for all } \varepsilon \in(0,1),
\end{align}
as long as $\gamma\in(1,3)$.
The Sobolev imbedding $W^{1,2}(\Omega) \hookrightarrow L^{\frac{2}{\gamma-1}}(\Omega)$ allows us to find $c_{1}=c_1(\Omega)>0$ such that
$$
\|c_\varepsilon(\cdot,t)\|_{L^{\frac{2}{\gamma-1}}(\Omega)}
 \leq c_{1}\left(\|\nabla \cep\|_{\Lom2} + \|\cep\|_{\Lom2} \right)
  \quad \text { for all } t \in(0, T)\text { and } \varepsilon \in(0,1).
$$
Substituting this in \eqref{u1} and using Young's inequality yields
$$
\begin{aligned}
&\quad\frac{1}{2} \frac{d}{d t} \int_{\Omega}c_\varepsilon^{2}+\int_{\Omega}c_\varepsilon^{2}
+\int_{\Omega}|\nabla c_\varepsilon|^{2} \\
& \leq c_{1}\Big(\|\nabla c_\varepsilon\|_{L^{2}(\Omega)}
+\| c_\varepsilon\|_{L^{2}(\Omega)}\Big)
\|n_\varepsilon\|_{L^{\frac{2}{3-\gamma}}(\Omega)} \\
& \leq \frac{1}{2} \int_{\Omega}|\nabla c_\varepsilon|^{2}
+\frac{1}{2} \int_{\Omega}c_\varepsilon^{2}
+c_1^2\|n_\varepsilon\|^2_{L^{\frac{2}{3-\gamma}}(\Omega)} \quad \text { in } (0, T) \text{ for all } \varepsilon \in(0,1),
\end{aligned}
$$
which directly tells us that
\begin{equation}
\frac{d}{d t} \int_{\Omega}c_\varepsilon^{2}
+\int_{\Omega}c_\varepsilon^{2}
+\int_{\Omega}|\nabla c_\varepsilon|^{2}
\leq 2c^2_1\|n_\varepsilon\|^2_{L^{\frac{2}{3-\gamma}}(\Omega)}
\quad \text { in } (0, T) \text { for all } \varepsilon \in(0,1).
\end{equation}
It can be deduced from the variation-of-constants formula that
\begin{equation}\label{u2}
\frac{d}{d t}\Big( {\e}^{t} \int_{\Omega}c_\varepsilon^{2}(x,t)dx\Big)
+{\e}^{t}\int_{\Omega}|\nabla c_\varepsilon|^{2}
\leq
2c^2_1{\e}^{t}\|n_\varepsilon\|^2_{L^{\frac{2}{3-\gamma}}(\Omega)}
\quad \text { in } (0, T) \text { for all } \varepsilon \in(0,1).
\end{equation}
For any $\tau\in(0,T]$, integrating \eqref{u2} with respect to the time-variable over the interval $(0,\tau)$ shows that
\begin{align*}
\int_{\Omega}c_\varepsilon^{2}(\cdot,\tau)
+{\e}^{-\tau}\int_{0}^{\tau}\int_{\Omega}{\e}^{ t}|\nabla c_\varepsilon(x,t)|^{2} dxdt
\leq
&2c^2_1\int_{0}^{\tau}\|n_\varepsilon(\cdot,t)\|^2_{L^{\frac{2}{3-\gamma}}(\Omega)} dt
+\|c_{0\varepsilon}\|^2_{L^2(\Omega)}\nonumber\\
\leq&
2c^2_1\int_{0}^{T}\|n_\varepsilon(\cdot,t)\|^2_{L^{\frac{2}{3-\gamma}}(\Omega)} dt
+\|c_{0\varepsilon}\|^2_{L^2(\Omega)}\quad \text { for all }\varepsilon \in(0,1).
\end{align*}
As $\gamma<2$, this combined with \eqref{c4'} and \eqref{initial1} guarantees our claims.
\end{proof}

\label{operator}
For the treatment of $c_{\eps }$, we need slightly more regularity than could be achieved from testing \eqref{qapp}$_2$ by $c_{\eps }$ as in the previous lemma. For the next testing procedure, we let $\A$ denote the operator $-\Delta +1 $ with Neumann boundary conditions, that is the sectorial operator defined by
\begin{align}\label{A1}
\A u:=-\Delta u+u \quad \text{ for}~u\in D(\A):=\{u\in
W^{2,2}(\Omega):\partial_{\nu}u=0~\text{on}~\partial\Omega\}
\end{align}
and briefly recall some of its properties:
\begin{lemma}\label{lem3.6}
\textrm{(i)} $\A$ has closed fractional
powers $\A^{\alpha}$ ($\alpha>0$) 
. They are self-adjoint.\\
\textrm{(ii)} For any $\alpha\in(0,1)$, the domain $D\left(\A^{\alpha}\right)$ is continuously embedded in $W^{2\alpha,2}(\Omega)$.\\
\textrm{(iii)} For all $\alpha,\beta,\delta\in\mathbb{R}$ satisfying $\beta<\alpha<\delta$, there is $C>0$ such that
\begin{align*}
\|\A^\alpha\varphi\|_{L^2(\Omega)}\le C\|\A^\delta\varphi\|^{\frac{\alpha-\beta}{\delta-\beta}}_{L^2(\Omega)}
\|\A^\beta\varphi\|^{\frac{\delta-\alpha}{\delta-\beta}}_{L^2(\Omega)}\quad\text{for all } \varphi\in D(\A^\delta).
\end{align*}
\end{lemma}
\begin{proof}
 This lemma is a summary of well-known properties of $\A$, which can be found in, e.g., \cite{He,Fr}, \cite[Theorem 6.7]{Fu1969} and \cite[Theorem 14.1]{Fr}.
\end{proof}

Before giving Lemma \ref{lem3.7}, we state the following elementary result allowing to derive boundedness by means of an argument of ordinary differential inequalites which we will employ in several places while studying the development of certain energy-like functionals.

\begin{lemma}\label{ODI}
Let $M_1,M_2>0$. Then there is $C=C(M_1,M_2)>0$ with the property that
whenever for some $T_*\in(0,∞]$ and $\tau\in[0,T_*)$ the function $y\in C^0([\tau,T_*))\cap C^1((\tau,T_*))$ is nonnegative and
satisfies
\begin{align}\label{ODI1}
y'(t)+h(t)\leq a(t) y(t)+b(t), \quad t\in(\tau,T_*),
\end{align}
where $h$, $a$, $b$ are nonnegative integrable functions with $a$ and $b$ satisfying
\begin{align}\label{ODIcondition}
\int_\tau^{T_*}a(t)dt\le M_1~~ and ~~\int_\tau^{T_*}b(t)dt\le M_2,
\end{align}
then it holds that
\begin{align}\label{ODIresult}
\sup_{t\in(\tau,T_*)}y(t)\le Cy(\tau)+C~~ and ~~\int_\tau^{T_*}h(t)dt\le Cy(\tau)+C.
\end{align}

\end{lemma}
\begin{proof}
Nonnegativity of $h$ on $(\tau,T_*)$ combined with \eqref{ODI1} ensures
\begin{align*}
y'(t)\leq a(t) y(t)+b(t), \quad t\in(\tau,T_*).
\end{align*}
An integration of this differential inequality gives that
for any $t\in(\tau, T_*)$,
\begin{align}\label{ODI2}
y(t)  \leq y\left(\tau\right) {\mathrm e}^{\int_{\tau}^{t} a(s) d s}
+\int_{\tau}^{t} \mathrm{e}^{\int_{s}^{t} a(\sigma) d \sigma} b(s) d s 
 \leq y(\tau) \mathrm{e}^{M_1}+\int_{ \tau }^{t} \mathrm{e}^{M_1} b(s) d s 
 \leq y(\tau) \mathrm{e}^{M_1}+M_2\mathrm{e}^{M_1}.
\end{align}
Integrating \eqref{ODI1} over $(\tau, T_*)$ and applying \eqref{ODI2} we infer that
\begin{align}\label{ODI3}
\int_{\tau}^{T_*}h(t)dt & \leq
\int_{\tau}^{T_*}a(t)y(t)dt+\int_{\tau}^{T_*}b(t)dt+ y\left(\tau\right)
 \nonumber \\
& \leq \big(y(\tau) \mathrm{e}^{M_1}+M_2\mathrm{e}^{M_1}\big)\int_{\tau}^{T_*}a(t)dt
+M_2 +y\left(\tau\right) \nonumber \\
& \leq y(\tau) \big( M_1\mathrm{e}^{M_1}+1)+M_1M_2\mathrm{e}^{M_1}+M_2.
\end{align}
With taking $C=(M_1+1)\max\{\mathrm{e}^{M_1},M_2\mathrm{e}^{M_1}\}$ in the claim of this lemma, we can end the proof by virtue of \eqref{ODI2} and \eqref{ODI3}.
\end{proof}

Lemma~\ref{ODI} will find its first application during the derivation of the following estimates for the second solution component.

\begin{lemma}\label{lem3.7}
Let $T>0$. If $\gamma\in(1,2)$, there exist positive constants $\beta>0$ and $C=C(T)>0$ such that
\begin{align}\label{v0}
&\int_{0}^{T}\int_{\Omega}|\A^{\frac{\beta+1}{2}} c_\varepsilon(x, t)|^{2}dxdt \leq C\quad \text { for all } \varepsilon \in(0,1).
\end{align}
\end{lemma}
\begin{proof}
With $\gamma>1$ and $\sigma>0$ (fixed at the beginning of Section~\ref{sec3}), we take $\beta=\min\{2\sigma,\gamma-1\}$ and obtain from \eqref{initial1} that
\begin{align}\label{v1}
&\int_{\Omega}|\A^{\frac{\beta}{2}}c_{0\varepsilon}|^{2} \leq c_1\quad \text { for all } \varepsilon \in(0,1)
\end{align}
with some $c_1=c_1(\sigma,\gamma,\Omega)>0$.
By testing $\eqref{qapp}_{2}$ by $\A^{\beta}c_\varepsilon$ and utilizing the self-adjointness
of $\A$ and its powers, it follows that for all $\varepsilon\in(0,1)$
\begin{equation}\label{v2}
\frac{1}{2}\frac{d}{d t} \int_{\Omega}|\A^{\frac{\beta}{2}}c_\varepsilon|^{2}
+\int_{\Omega}|\A^{\frac{\beta+1}{2}}c_\varepsilon|^{2}
=-\int_{\Omega} (\A^{\beta}c_\varepsilon) \nabla c_\varepsilon\cdot\mathbf{u}_\varepsilon
+\int_{\Omega} \frac{n_\varepsilon \A^{\beta}c_\varepsilon}{1+\varepsilon n_\varepsilon}
\quad \text { in } (0, T).
\end{equation}
An application of H\"{o}lder's inequality implies that
\begin{equation}\label{v3}
-\int_{\Omega} \A^{\beta}c_\varepsilon \nabla c_\varepsilon\cdot\mathbf{u}_\varepsilon
\le \| \A^{\beta}c_\varepsilon\|_{L^{4}(\Omega)}
 \| \nabla c_\varepsilon\|_{L^{2}(\Omega)}
  \| \mathbf{u}_\varepsilon\|_{L^{4}(\Omega)}
\quad \text { in } (0, T)\text { for all } \varepsilon \in(0,1).
\end{equation}
Utilizing the Sobolev imbedding inequality (of \cite[Theorem 6.7]{NPV}) and Lemma \ref{lem3.6}, we can see
\begin{align}\label{v4}
\| \A^{\beta}c_\varepsilon\|_{L^{4}(\Omega)}
&\le c_2\|\A^{\beta}c_\varepsilon\|_{W^{\frac{1}{2}, 2}(\Omega)}\nonumber\\
&\le c_3\|\A^{\beta+\frac{1}{4}}c_\varepsilon\|_{L^{2}(\Omega)}\nonumber\\
&\le c_4\|\A^{\frac{\beta+1}{2}}c_\varepsilon\|^{\beta+\frac{1}{2}}_{L^{2}(\Omega)}
\|\A^{\frac{\beta}{2}}c_\varepsilon\|^{1-\beta-\frac{1}{2}}_{L^{2}(\Omega)}
\quad \text { in }(0, T)\text { for all } \varepsilon \in(0,1)
\end{align}
with positive constants $c_2,c_3,c_4$ only depending on $\sigma,\gamma$ and $\Omega$.
Similarly, we also have $c_5,c_6>0$ such that
\begin{align}\label{v5}
\| \nabla c_\varepsilon\|_{L^{2}(\Omega)}
&\le c_5\|\A^{\frac{1}{2}}c_\varepsilon\|_{L^{2}(\Omega)}\nonumber\\
&\le c_6\|\A^{\frac{\beta+1}{2}}c_\varepsilon\|^{1-\beta}_{L^{2}(\Omega)}
\|\A^{\frac{\beta}{2}}c_\varepsilon\|^{\beta}_{L^{2}(\Omega)}
\quad \text { in }(0, T)\text { for all } \varepsilon \in(0,1).
\end{align}
Inserting \eqref{v4} and \eqref{v5} into \eqref{v3} and applying Young's inequality give that
\begin{align}
-\int_{\Omega} \A^{\beta}c_\varepsilon \nabla c_\varepsilon\cdot\mathbf{u}_\varepsilon
&\le c_4c_6\|\A^{\frac{\beta+1}{2}}c_\varepsilon\|^{\frac{3}{2}}_{L^{2}(\Omega)}
\|\A^{\frac{\beta}{2}}c_\varepsilon\|^{\frac{1}{2}}_{L^{2}(\Omega)}
 \| \mathbf{u}_\varepsilon\|_{L^{4}(\Omega)}\nonumber\\
&\le\frac{1}{4}\|\A^{\frac{\beta+1}{2}}c_\varepsilon\|^2_{L^{2}(\Omega)}
+\frac{27}{4}c^4_4c^4_6\|\A^{\frac{\beta}{2}}c_\varepsilon\|^{2}_{L^{2}(\Omega)}
 \| \mathbf{u}_\varepsilon\|^4_{L^{4}(\Omega)}
\end{align}
in $(0, T)$ for all $\varepsilon \in(0,1)$.
Since $\beta\le\gamma-1$, Lemma \ref{lem3.6} (ii) ensures the existence of $c_7,c_8>0$ satisfying
\begin{align}\label{v6}
\|\A^{\beta}c_\varepsilon\|_{L^{\frac{2}{\gamma-1}}(\Omega)}
\le c_7\|\A^{\beta}c_\varepsilon\|_{W^{2-\gamma,2}(\Omega)}
\le
c_8\|\A^{\frac{\beta+1}{2}}c_\varepsilon\|_{L^{2}(\Omega)} \text{ in } (0,T) \text{ for all } \eps \in(0,1).
\end{align}
Thus, it can be obtained by the H\"{o}lder inequality (due to $γ\in(1,3)$ ensuring $\frac{2}{3-γ}>1$ and $\frac{2}{γ-1}>1$)
and Young's inequality that
\begin{align*}
-\int_{\Omega}n_\varepsilon \A^{\beta}c_{\varepsilon }
 &\le
 \| n_\varepsilon\|_{L^{\frac{2}{3-\gamma}}(\Omega)}
 \|\A^{\beta}c_\varepsilon\|_{L^{\frac{2}{\gamma-1}}(\Omega)}\nonumber\\
   &\le
   c_8^2\| n_\varepsilon\|^2_{L^{\frac{2}{3-\gamma}}(\Omega)}+
 \frac{1}{4c_8^2}
 \|\A^{\beta}c_\varepsilon\|^2_{L^{\frac{2}{\gamma-1}}(\Omega)}\nonumber\\
   &\le
   c_8^2\| n_\varepsilon\|^2_{L^{\frac{2}{3-\gamma}}(\Omega)}+
 \frac{1}{4}
 \|\A^{\frac{\beta+1}{2}}c_\varepsilon\|^2_{L^{2}(\Omega)}
 ~~\text { in }(0, T)\text { for all } \varepsilon \in(0,1).
\end{align*}
Together with \eqref{v2}, \eqref{v5} and \eqref{v6}, this shows that
\begin{align}\label{v7}
&\quad\frac{d}{d t} \int_{\Omega}|\A^{\frac{\beta}{2}}c_\varepsilon|^{2}
+\int_{\Omega}|\A^{\frac{\beta+1}{2}}c_\varepsilon|^{2} \nonumber\\
&\le c_{9}\| n_\varepsilon\|^2_{L^{\frac{2}{3-\gamma}}(\Omega)}
+c_{9}\|\A^{\frac{\beta}{2}}c_\varepsilon\|^{2}_{L^{2}(\Omega)}
 \| \mathbf{u}_\varepsilon\|^4_{L^{4}(\Omega)}
~~\text { in } (0, T)\text { for all } \varepsilon \in(0,1)
\end{align}
with $c_{9}=\max\{2c_8^2,\frac{27}2c_4^4c_6^4\}>0$. In order to apply Lemma~\ref{ODI} with
\begin{align*}
y(t)&=\norm[\Lom2]{\A^{\frac{\beta}2}\cep(\cdot,t)}, \qquad& h(t)&=\norm[\Lom2]{\A^{\frac{\beta+1}2}\cep(\cdot,t)},\quad\\ a(t)&=c_9\norm[\Lom4]{\uep(\cdot,t)}^4,& b(t)&=c_9\norm[\Lom{\frac{2}{3-γ}}]{\nep(\cdot,t)},&&\qquad t\in(0,T),
\end{align*}
and $\tau=0$, we note that Corollary~\ref{cor2} and Corollary~\ref{cor1} ensure \eqref{ODIcondition}, as long as $γ\in(1,2)$, and by \eqref{v1}, we may conclude \eqref{v0} from \eqref{ODIresult}.
\end{proof}
The final outcome of the previous bounds on $\cep$ is summarized in the following compactness statement, which will be directly applicable in the convergence proofs in Section~\ref{sec:convergence}.
\begin{lemma}\label{lem3.8}
Assume $\gamma\in(1,2)$. Let $T>0$. Then
$\{c_\varepsilon\}_{\varepsilon\in(0,1)}$ is relatively compact with respect to the strong topology in $L^2\big((0, T);W^{1,2}(\Omega)\big)$.
\end{lemma}
\begin{proof}
Multiplying the second equation in \eqref{qapp} by $\xi \in C_0^\infty(\Omega)$
and integrating by parts show that
\begin{align}\label{e1}
\left|\int_{\Omega} c_{\varepsilon t}(\cdot, t)\xi\right| &=\left|-
\int_{\Omega}\nabla c_\varepsilon\cdot\mathbf{u}_\varepsilon \xi
-\int_{\Omega}\nabla c_\varepsilon\cdot\nabla \xi-\int_{\Omega} c_\varepsilon\xi
+\int_{\Omega} \frac{n_{\varepsilon}}{1+\varepsilon n_{\varepsilon}}\xi \right|\nonumber \\
 & \leq
\left\|\nabla c_\varepsilon(\cdot, t)\right\|_{L^{2}(\Omega)}\|\mathbf{u}_\varepsilon\|_{L^{2}(\Omega)}\|\xi\|_{L^\infty(\Omega)}
\quad+\left\| c_\varepsilon(\cdot, t)\right\|_{L^{2}(\Omega)}\|\xi\|_{L^{2}(\Omega)}\nonumber
\\&\quad+ \left\|\nabla c_\varepsilon(\cdot, t)\right\|_{L^{2}(\Omega)}\|\nabla \xi\|_{L^{2}(\Omega)}
 +\left\|n_{\varepsilon}(\cdot, t)\right\|_{L^{1}(\Omega)}
\|\xi\|_{L^{\infty}(\Omega)}
\end{align}
We deduce from \eqref{e1} and the embedding $W_0^{1,3}(\Omega)\hookrightarrow L^\infty(\Omega)\cap W^{1,2}_0(\Omega)$ that there is $c_1>0$ satisfying
$$
\left\|c_{\varepsilon t}(\cdot, t)\right\|_{\big(W_0^{1,3}(\Omega)\big)^{*}}
\leq c_1\Big(\left\|c_{\varepsilon}(\cdot, t)\right\|_{W^{1,2}(\Omega)}
+\|n_{\varepsilon}(\cdot, t)\|_{L^{1}(\Omega)}\Big)\quad \text { for all } t\in(0,T)\text { and } \varepsilon \in(0,1),
$$
which along with Lemma \ref{lem3.7} and \eqref{c2} leads to the observation that
\begin{align}\label{e2}
\left\{c_{\varepsilon t}\right\}_{\varepsilon \in(0,1)} \text { is uniformly bounded in } L^{2}\Big((0, T);\big(W_0^{1,3}(\Omega)\big)^{*}\Big).
\end{align}
Since $D(\A^{\frac{\beta+1}{2}})$ is compactly embedded into $W^{1,2}(\Omega)$ with $\beta>0$ determined in Lemma \ref{lem3.7},
the claim results from \eqref{v0}, \eqref{e2} and the Aubin-Lions lemma.\end{proof}
\subsection{Estimates for \texorpdfstring{$\{n_{\varepsilon}\}_{\varepsilon\in(0,1)}$}{\{n\_ε \}}}
We begin this subsection with the uniform integrability involving $n_{\varepsilon}$.
\begin{lemma}\label{lem3.9}
Let $T>0$. Then
\begin{align}\label{k0}
 \{n_{\varepsilon}\}_{\varepsilon \in(0,1)}~\text{is~ uniformly ~integrable~ over~ }\Omega \times(0, T).
\end{align}
Moreover, we have that both
\begin{align}\label{k0'}
\left\{\left(n_{\varepsilon}+1\right)^{-1}n^2_{\varepsilon}\right\}_{\varepsilon \in(0,1)}
~~\text{and }~~\left\{\left(n_{\varepsilon}+1\right)^{-1} f\left(n_{\varepsilon}\right) \right\}_{\varepsilon \in(0,1)}\text{are uniformly integrable over $\Omega \times(0, T)$.}
\end{align}

\end{lemma}
\begin{proof}
We let $g(z)=z$, $g(z)=\f{z^2}{1+z}$ or $g(z)=\f{|f(z)|}{1+z}$ for the proofs of \eqref{k0} or the first or second part of \eqref{k0'}, respectively. In each of these cases, there is $c_1>0$ such that
\[
 |g(z)| = g(z) \le c_1+c_1z^{\max\{1,\gamma-1\}} \qquad \text{for every } z\ge 0
\]
and according to Lemma~\ref{lem3.2}, we therefore can find $c_2>0$ satisfying
\[
 \int_0^T\io |g(\nep)|^{\frac{\gamma}{\max\{1,\gamma-1\}}} \le c_2 \quad\text{ for every }\eps \in(0,1),
\]
by the de la Vallée-Poussin theorem proving uniform integrability of $\{g(\nep)\}_{\eps \in(0,1)}$.
\end{proof}

In order to conclude $L^1$-convergence from uniform integrability by means of Vitali's convergence theorem, we will additionally require some convergence in measure. Aiming to obtain this along a subsequence obtained from application of an Aubin-Lions lemma to $\ln(\nep+1)$, we thus prepare the following estimates of derivatives of the latter.

\begin{lemma}\label{lem3.10}
Let $T>0$ and $γ\in(1,2)$. Then we can find $C=C(T)>0$ such that
\begin{align}\label{y0}
\int_{0}^{T} \int_{\Omega}\left|\nabla\ln\left(n_{\varepsilon}+1\right)\right|^{2} \leq C \quad \text { for all } \varepsilon \in(0,1)
\end{align}
and
\begin{align}\label{y0'}
\int_{0}^{T}
\big\|\partial_{t}\ln\left(n_{\varepsilon}(\cdot, t)+1\right)\big\|_{\left(W^{1, 3}(\Omega)\right)^{*}} d t \leq C \quad \text { for all } \varepsilon \in(0,1).
\end{align}
\end{lemma}
\begin{proof}
We choose $ \varphi\equiv1 $ in Lemma \ref{lem2.2} and integrate \eqref{a0} with respect to $t\in(0, T)$ to get
\begin{align}\label{y1}
\quad\int_{0}^{T} \int_{\Omega}\left|\nabla\ln\left(n_{\varepsilon}+1\right)\right|^{2}
&=\int_{\Omega}\ln\big(n_{\varepsilon}(\cdot, T)+1\big)-\int_{\Omega}\ln\left(n_{0\varepsilon}+1\right)
+\int_{0}^{T} \int_{\Omega}n_\varepsilon\left(n_{\varepsilon}+1\right)^{-2}
\nabla n_\varepsilon \cdot\nabla c_{\varepsilon}\nonumber\\
&\quad-\int_{0}^{T} \int_{\Omega}\left(n_{\varepsilon}+1\right)^{-1}f(n_\varepsilon)
+\varepsilon \int_{0}^{T} \int_{\Omega}\left(n_{\varepsilon}+1\right)^{-1}
n^2_{\varepsilon}
\quad\text{for~all~}\varepsilon\in(0,1),
\end{align}
where the term involving $\nabla \cep$ can be estimated due to $\f{\nep}{1+\nep}\le 1$ and by Young's inequality according to
\begin{align}\label{y2}
 &\int_{0}^{T} \int_{\Omega}n_\varepsilon\left(n_{\varepsilon}+1\right)^{-2}
\nabla n_\varepsilon \cdot\nabla c_{\varepsilon}
\le& \frac{1}{2}\int_{0}^{T} \int_{\Omega}\left|\nabla\ln\left(n_{\varepsilon}+1\right)\right|^{2}
+\frac{1}{2}\int_{0}^{T} \int_{\Omega}\left|\nabla c_{\varepsilon}\right|^{2}
\quad\text{for~all~}\varepsilon\in(0,1).
\end{align}
Lemma \ref{lem3.1} yields a constant $c_1>0$ such that
\begin{align}\label{y3}
-\int_{0}^{T} \int_{\Omega}\left(n_{\varepsilon}+1\right)^{-1} f\left(n_{\varepsilon}\right)
\le \int_{0}^{T} \int_{\Omega}|f( n_{\varepsilon})|\le c_1
\quad\text{for~all~}\varepsilon\in(0,1)
\end{align}
and
\begin{align}\label{y3'}
\varepsilon\int_{0}^{T} \int_{\Omega}\left(n_{\varepsilon}+1\right)^{-1} n^2_{\varepsilon}
\le \int_{0}^{T} \int_{\Omega} n_{\varepsilon}\le c_1
\quad\text{for~all~}\varepsilon\in(0,1).
\end{align}
Inserting \eqref{y2}-\eqref{y3'} into \eqref{y1} and invoking Lemma \ref{lem3.5},
we arrive at the claim \eqref{y0}.

We now turn to the assertion \eqref{y0'}.
Taking $\xi\in C^{\infty}(\Ombar)$, we may invoke Lemma~\ref{lem2.2} for the function $\varphi$ defined by $\varphi(\cdot,t)=\xi$ for all $t>0$, obtaining
\begin{align}
&\quad\left|\int_{\Omega}\partial_{t}\ln\left(n_{\varepsilon}(\cdot, t)+1\right)\xi\right|\nonumber\\
 &=\Big|\int_{\Omega} (n_\varepsilon+1)^{-2}|\nabla n_\varepsilon|^2\xi
 -\int_{\Omega} (n_\varepsilon+1)^{-1}\nabla n_\varepsilon\cdot\nabla\xi
 \nonumber\\
 &\quad+\int_{\Omega}\ln(n_\varepsilon+1) \mathbf{u}_\varepsilon\cdot\nabla \xi
 -\int_{\Omega} (n_\varepsilon+1)^{-2} n_\varepsilon
 \nabla n_\varepsilon\cdot\nabla c_\varepsilon\xi
 \nonumber\\
  &\quad+\int_{\Omega} (n_\varepsilon+1)^{-1} n_\varepsilon
  \nabla c_\varepsilon\cdot\nabla\xi
 +\int_{\Omega} (n_\varepsilon+1)^{-1} f(n_\varepsilon)\xi
 -\varepsilon\int_{\Omega} (n_\varepsilon+1)^{-1} n_\varepsilon^2\xi
 \Big|\nonumber\\
&\le \left\|\nabla\ln\left(n_{\varepsilon}+1\right)(\cdot, t)\right\|^2_{L^{2}(\Omega)}\|\xi\|_{L^{\infty}(\Omega)}\nonumber
+\left\|\nabla\ln\left(n_{\varepsilon}+1\right)(\cdot, t)\right\|_{L^{2}(\Omega)}
\|\nabla\xi\|_{L^{2}(\Omega)}\\
&\quad+
\left\|\ln\left(n_{\varepsilon}+1\right)(\cdot, t)\right\|_{L^{4}(\Omega)}\left\|\mathbf{u}_\varepsilon(\cdot, t)\right\|_{L^{4}(\Omega)}\|\nabla \xi\|_{L^{2}(\Omega)}\\
&\quad +\left\|\nabla\ln\left(n_{\varepsilon}+1\right)(\cdot, t)\right\|_{L^{2}(\Omega)}
\left\|\nabla c_\varepsilon(\cdot, t)\right\|_{L^{2}(\Omega)}\|\xi\|_{L^{\infty}(\Omega)}\nonumber\\
&\quad+\left\|\nabla c_\varepsilon(\cdot, t)\right\|_{L^{2}(\Omega)}\|\nabla\xi\|_{L^{2}(\Omega)}\nonumber
+\left\|f(n_\varepsilon(\cdot,t))\right\|_{L^{1}(\Omega)}
\|\xi\|_{L^{\infty}(\Omega)}
+\left\|n_\varepsilon(\cdot, t)\right\|_{L^{1}(\Omega)}\|\xi\|_{L^{\infty}(\Omega)}
\end{align}
for all $t\in(0,T)$ and $\varepsilon\in(0,1)$, where the second inequality relies on $\frac{n_{\eps }}{1+n_{\eps }}\le1$. This in conjunction with Young's inequality,
Poincar\'{e}'s inequality and the Sobolev embedding inequality ensures the existence of
$c_3=c_3(\Omega)>0$ such that
\begin{align*}
\left|\int_{\Omega}\partial_{t}\ln\left(n_{\varepsilon}(\cdot, t)+1\right)\xi\right|
 \leq c_3
 \Big(1+\left\|\mathbf{u}_\varepsilon(\cdot, t)\right\|^2_{L^{4}(\Omega)}
& +\left\|\nabla\ln\left(n_{\varepsilon}+1\right)(\cdot, t)\right\|^2_{L^{2}(\Omega)}
+\left\|\nabla c_\varepsilon(\cdot, t)\right\|^2_{L^{2}(\Omega)}\nonumber\\
&\quad
+\left\|f(n_\varepsilon(\cdot, t))\right\|_{L^{1}(\Omega)}
+\left\|n_\varepsilon(\cdot, t)\right\|^2_{L^{1}(\Omega)}\Big)\|\xi\|_{ W^{1,3}(\Omega) }
\end{align*}
for all $t\in(0,T)$ and $\varepsilon\in(0,1)$, which directly gives that
\begin{align}\label{y4}
&\hspace*{-1cm}\left\|\partial_{t}\ln\left(n_{\varepsilon}(\cdot, t)+1\right)\right\|_{\left(W^{1, 3}(\Omega)\right)^{*}}\nonumber\\
 &  \leq c_3
 \Big(1+\left\|\mathbf{u}_\varepsilon(\cdot, t)\right\|^2_{L^{4}(\Omega)}
+\left\|\nabla\ln\left(n_{\varepsilon}+1\right)(\cdot, t)\right\|^2_{L^{2}(\Omega)}
+\left\|\nabla c_\varepsilon(\cdot, t)\right\|^2_{L^{2}(\Omega)}\nonumber\\
&\quad
+\left\|f(n_\varepsilon(\cdot, t))\right\|_{L^{1}(\Omega)}
+\left\|n_\varepsilon(\cdot, t)\right\|_{L^{1}(\Omega)}\Big)
\quad\text{for~all~}t\in(0,T)~\text{and}~\varepsilon\in(0,1).
\end{align}
By integrating \eqref{y4} in time,
we can derive the desired estimate from \eqref{y0} and Lemma \ref{lem3.1} together with Lemmata \ref{lem3.3} and \ref{lem3.7}.
\end{proof}
\begin{lemma}\label{lem3.11}
Let $T>0$ and assume $γ\in(1,2)$. Then there is $C=C(T)>0$ such that
\begin{align}\label{s0}
\int_{0}^{T} \int_{\Omega}\left|\nabla\left(n_{\varepsilon}+1\right)^{-1}\right|^{2} \leq C \quad \text { for all } \varepsilon \in(0,1).
\end{align}
\end{lemma}
\begin{proof}
According to Lemma~\ref{lem:testing} applied to $F(s)=\f1s$,
\begin{align}\label{s5}
\int_{\Omega}\partial_{t}\left(n_{\varepsilon}(\cdot, t)+1\right)^{-1}\varphi
 =& -2\int_{\Omega} (n_\varepsilon+1)^{-2}|\nabla n_\varepsilon|^2\varphi
 +\int_{\Omega} (n_\varepsilon+1)^{-2}\nabla n_\varepsilon\cdot\nabla\varphi\nonumber\\
&\quad+ \int_{\Omega}(n_\varepsilon+1)^{-1} \mathbf{u}_\varepsilon\cdot\nabla \varphi +2\int_{\Omega} (n_\varepsilon+1)^{-3} n_\varepsilon
 \nabla n_\varepsilon\cdot\nabla c_\varepsilon\varphi \nonumber\\
  &\quad-\int_{\Omega} (n_\varepsilon+1)^{-2} n_\varepsilon
  \nabla c_\varepsilon\cdot\nabla\varphi
 -\int_{\Omega} (n_\varepsilon+1)^{-2} f(n_\varepsilon)\varphi\nonumber\\
  &\quad+\varepsilon\int_{\Omega}\left(n_{\varepsilon}+1\right)^{-2}
  n^2_{\varepsilon}\varphi
  \quad\text{ in }(0,T)~\text{and}~\varepsilon\in(0,1).
\end{align}
We first take $\varphi\equiv1$,
integrate \eqref{s5} with respect to the time-variable and use Young's inequality to obtain that
\begin{align*}
2\int_0^T\int_{\Omega}(n_\varepsilon+1)^{-2} |\nabla n_\varepsilon|^2
&\le\int_{\Omega}\left(n_{0\varepsilon}(\cdot)+1\right)^{-1}
 +2\int_0^T\int_{\Omega} (n_\varepsilon+1)^{-3} n_\varepsilon
 \nabla n_\varepsilon\cdot\nabla c_\varepsilon \nonumber\\
  &\quad
 -\int_0^T\int_{\Omega} (n_\varepsilon+1)^{-2} f(n_\varepsilon)
 +\varepsilon\int_0^T\int_{\Omega}\left(n_{\varepsilon}+1\right)^{-2}
  n^2_{\varepsilon}\nonumber\\
  &\le\int_{\Omega}\left(n_{0\varepsilon}(\cdot)+1\right)^{-1}
 +\int_0^T\int_{\Omega} (n_\varepsilon+1)^{-2} |\nabla n_\varepsilon|^2
 +\int_0^T\int_{\Omega} |\nabla c_\varepsilon|^2 \nonumber\\
  &\quad
 -\int_0^T\int_{\Omega} (n_\varepsilon+1)^{-2} f(n_\varepsilon)
 +\varepsilon\int_0^T\int_{\Omega}\left(n_{\varepsilon}+1\right)^{-2}
  n^2_{\varepsilon}
  ~~\text{for~all~} \varepsilon\in(0,1).
\end{align*}
This in conjunction with \eqref{u0'} and \eqref{c1} shows that
\begin{align}\label{s5'}
\int_0^T\int_{\Omega}(n_\varepsilon+1)^{-2} |\nabla n_\varepsilon|^2
  &\le\int_{\Omega}\left(n_{0\varepsilon}(\cdot)+1\right)^{-1}
 +\int_0^T\int_{\Omega} |\nabla c_\varepsilon|^2
  +\int_0^T\int_{\Omega} |f(n_\varepsilon)|\nonumber\\
  &\quad
+\varepsilon\int_0^T\int_{\Omega}\left(n_{\varepsilon}+1\right)^{-2}
  n^2_{\varepsilon}
  \nonumber\\
  &\le c_1
~~\text{for~all~} \varepsilon\in(0,1)
\end{align}
with some $c_1=c_1(T)>0$.
\end{proof}

\section{Global existence of generalized solution. Proof of Theorem~\ref{th 1}}
\label{sec:convergence}
This section is devoted to proving the global existence of a generalized solution to (\ref{q1}).
To achieve this goal, we shall provide necessary
convergence properties for all components in the two succeeding lemmata.
 First, the next result concerned with $\{c_\varepsilon\}_{\varepsilon\in(0,1)}$ and $\{\mathbf{u}_\varepsilon\}_{\varepsilon\in(0,1)}$ directly follows from lemmata of the previous section:

\begin{lemma}\label{lem4.1}
If $\gamma\in(1,2)$, there exist $\left\{\varepsilon_{j}\right\}_{j \in \mathbb{N}} \subset(0,1)$ as well as functions
\begin{equation}
\left\{\begin{array}{l}{\mathbf{u} \in L_{loc}^{2}\left([0, \infty) ; W^{1,2}(\Omega;\mathbb{R}^2)\right)\quad} \\
{c\in L_{l o c}^{2}\left([0, \infty) ; W^{1,2}(\Omega)\right)}\end{array}\right.
\end{equation}
with $c\ge 0$ a.e. in $\Omega\times(0,\infty)$
such that $\varepsilon_{j} \searrow 0$ as $j \rightarrow \infty$, and
\begin{align}
 c_{\varepsilon} &\rightarrow c &&\text { in } L_{l o c}^{2}\big(\overline{\Omega}\times[0, \infty)\big) \text { and a.e. in } \Omega \times(0, \infty),\label{ae1}\\
\nabla c_{\varepsilon}& \rightarrow \nabla c && \text { in } L^{2}_{l o c}\big(\overline{\Omega}\times[0, \infty)\big) \text { and a.e. in } \Omega \times(0, \infty),\label{ae2}\\
 \mathbf{u}_{\varepsilon} &\rightarrow \mathbf{u} &&\text { in } L_{l o c}^{2}\big(\overline{\Omega}\times[0, \infty)\big) \text { and a.e. in } \Omega \times(0, \infty),\label{ae3}\\
\nabla \mathbf{u}_{\varepsilon}& \rightharpoonup \nabla \mathbf{u} && \text { in } L^{2}_{l o c}\big(\overline{\Omega}\times[0, \infty)\big) \text { and a.e. in } \Omega \times(0, \infty)\label{ae4}
\end{align}
as $\varepsilon=\varepsilon_{j} \searrow 0$.
\end{lemma}
\begin{proof}
 The relative compactness of $\{c_\varepsilon\}_{\varepsilon\in(0,1)}$ in $L^2((0,T);W^{1,2}(\Omega))$ assured in Lemma~\ref{lem3.8} for every $T>0$ guarantees the existence of a sequence along which $\cep$ and $\nabla\cep$ converge in $L^2_{loc}(\Ombar\times[0,\infty))$ (implying a.e. convergence along a further subsequence), and the bounds from Lemma \ref{lem3.3} can be used to conclude \eqref{ae3} and \eqref{ae4} along a suitable sequence.
\end{proof}

Regarding convergence of the first component, we note the following:

\begin{lemma}\label{lem4.2}
There exist $\left\{\varepsilon_{j}\right\}_{j \in \mathbb{N}} \subset(0,1)$ and a function
\begin{align}
n\in L_{l o c}^{1}\big(\overline{\Omega}\times[0, \infty)\big)
\end{align}
such that $\{\varepsilon_{j}\}_{j\inℕ}$ is a subsequence of the sequence found in Lemma~\ref{lem4.1}, and 
\begin{align}
n_{\varepsilon}& \rightarrow n \text { in } L^{1}_{l o c}\big(\overline{\Omega}\times[0, \infty)\big) \text { and a.e. in } \Omega \times(0, \infty)\label{L1}.
\end{align}
Moreover, we have
\begin{align}
\ln\left(n_{\varepsilon}+1\right)
&\rightarrow
\ln\left(n+1\right)&& \text { in } L^{2}_{l o c}\big(\overline{\Omega}\times[0, \infty)\big),
\label{L2}\\
\nabla\ln\left(n_{\varepsilon}+1\right)
&\rightharpoonup
\nabla\ln\left(n+1\right)&& \text { in } L^{2}_{l o c}\big(\overline{\Omega}\times[0, \infty)\big),
\label{L3}\\
\nabla\left(n_{\varepsilon}+1\right)^{-1}
&\rightharpoonup
\nabla\left(n+1\right)^{-1}&& \text { in } L^{2}_{l o c}\big(\overline{\Omega}\times[0, \infty)\big),
\label{L3'}\\
(n_\varepsilon+1)^{-1}f(n_\varepsilon)
&\rightarrow
(n+1)^{-1}f(n)&& \text { in } L^{1}_{l o c}\big(\overline{\Omega}\times[0, \infty)\big),\label{L4}
\\
(n_\varepsilon+1)^{-1}n^2_\varepsilon
&\rightarrow
(n+1)^{-1}n^2&& \text { in } L^{1}_{l o c}\big(\overline{\Omega}\times[0, \infty)\big)\label{L4'}
\end{align}
as $\varepsilon=\varepsilon_{j} \searrow 0$.
\end{lemma}
\begin{proof}
It follows by Lemma \ref{lem3.10} that with some $c_1=c_1(T)>0$,
\begin{align}\label{i1'}
\int_{0}^{T} \int_{\Omega}\left|\nabla\ln\left(n_{\varepsilon}+1\right)\right|^{2}
+\int_{0}^{T}
\left\|\partial_{t}\ln\left(n_{\varepsilon}+1\right)\right\|_{\left(W^{1, 3}(\Omega)\right)^{*}} \le c_1
\end{align}
holds for every $\varepsilon\in(0,1)$.
The Aubin-Lions lemma enables us to find a function $g\in L^2_{loc}\big([0,\infty); W^{1,2}(\Omega)\big)$ and a subsequence of $\{\eps _j\}_{j\inℕ}$ from Lemma~\ref{lem4.1} (which we do not relabel) satisfying
\begin{align}\label{i2}
\ln\left(n_{\varepsilon}+1\right)
&\rightarrow
g, \quad \nabla\ln\left(n_{\varepsilon}+1\right)
\rightharpoonup
\nabla g~~\text { in } L^{2}_{l o c}\big(\overline{\Omega}\times[0, \infty)\big),
\end{align}
as $\eps =\eps _j\to 0$ and, along a subsequence, $\ln(n_{\eps }+1)\to g$ a.e. in $\Omega\times(0,\infty)$. Letting $n=\e^{g}-1$, we clearly have that $n_{\eps }\to n$ a.e. in $\Omega\times(0,\infty)$. According to the Vitali convergence theorem,
we can apply Lemma \ref{lem3.9} to derive that furthermore
\begin{align}\label{i1}
n_{\varepsilon}& \rightarrow n \text { in } L^{1}_{l o c}\big(\overline{\Omega}\times[0, \infty)\big).
\end{align}
With the estimate \eqref{s0} at our disposal, we conclude that \eqref{L3'} holds along a subsequence.
Moreover, the almost everywhere convergence along a further subsequence, as entailed by \eqref{i1} combined with the continuity of $f$ ensures that
\begin{align}\label{i3}
(n_\varepsilon+1)^{-1}f(n_\varepsilon)
&\rightarrow
(n+1)^{-1}f(n_\varepsilon)~~\text { a.e. in } \Omega \times(0, \infty).
\end{align}
In view of \eqref{k0'} and \eqref{i3}, the Vitali convergence theorem
guarantees the assertion \eqref{L4}.
Based on the uniform integrability of $(n_\varepsilon+1)^{-1}n^2_\varepsilon$,
we can obtain \eqref{L4'} by the same reasoning.
\end{proof}

Consequences for certain ``mixed terms'' that appear in the definition of solutions are as follows:

\begin{lemma}\label{lem4.3}Assume that $\gamma\in(1,2)$.
Let $n,c,\mathbf{u}$ be given in Lemma \ref{lem4.1} and Lemma \ref{lem4.2}.
There exist $\left\{\varepsilon_{j}\right\}_{j \in \mathbb{N}} \subset(0,1)$ such that
\begin{align}
(n_\varepsilon+1)^{-2} n_\varepsilon\nabla n_\varepsilon\cdot\nabla c_\varepsilon
&\rightharpoonup
(n+1)^{-2} n\nabla n\cdot\nabla c&&
\text { in }
  L^{1}_{l o c}\big(\overline{\Omega}\times[0, \infty)\big),
\label{L5}\\
(n_\varepsilon+1)^{-1} n_\varepsilon\nabla c_\varepsilon
&\rightarrow(n+1)^{-1} n\nabla c
  && \text { in } L^{1}_{l o c}\big(\overline{\Omega}\times[0, \infty)\big),
\label{L6}\\
\ln\left(n_{\varepsilon}+1\right) \mathbf{u}_\varepsilon
&\rightarrow
\ln\left(n+1\right)\mathbf{u}
&& \text { in } L^{1}_{l o c}\big(\overline{\Omega}\times[0, \infty)\big),\label{L7}
 \\
c_\varepsilon\mathbf{u}_\varepsilon
&\rightarrow
 c\mathbf{u}
&& \text { in } L^{1}_{l o c}\big(\overline{\Omega}\times[0, \infty)\big),\label{L8}
 \\
 \mathbf{u}_\varepsilon\otimes\mathbf{u}_\varepsilon
&\rightarrow
 \mathbf{u}\otimes\mathbf{u}
&&
\text { in } L^{1}_{l o c}\big(\overline{\Omega}\times[0, \infty)\big)\label{L9}
\end{align}
as $\varepsilon=\varepsilon_{j} \searrow 0$.
\end{lemma}
\begin{proof}
It can be verified that
\begin{align}
(n_\varepsilon+1)^{-2} n_\varepsilon\nabla n_\varepsilon
=\nabla\ln\left(n_{\varepsilon}+1\right)
+\nabla (n_\varepsilon+1)^{-1},
\end{align}
which in conjunction with \eqref{L3} and \eqref{L3'} gives that
\begin{align}
(n_\varepsilon+1)^{-2} n_\varepsilon\nabla n_\varepsilon
\rightharpoonup
\nabla\ln\left(n+1\right)
+\nabla (n+1)^{-1}
\end{align}
in $L^{2}_{l o c}\big(\overline{\Omega}\times[0, \infty)\big)$
as $\varepsilon=\varepsilon_{j} \searrow 0$.
This with \eqref{ae2} directly guarantees \eqref{L5}.
Since $(n_\varepsilon+1)^{-1} n_\varepsilon\le1$
in $\Omega\times(0,\infty)$ for all $\varepsilon\in(0,1)$ and
$n_{\varepsilon} \rightarrow n $ a.e. in $\Omega \times(0, \infty)$, the
dominated convergence theorem tells us that as $\varepsilon=\varepsilon_{j} \searrow 0$,
\begin{align}\label{o1}
(n_\varepsilon+1)^{-1} n_\varepsilon
\rightarrow(n+1)^{-1} n~~\text { in } L^{2}_{l o c}\big(\overline{\Omega}\times[0, \infty)\big).
\end{align}
Hence, \eqref{L6} follows by a combination of
\eqref{o1} and \eqref{ae2}.
The assertion \eqref{L7} can be obtained by using \eqref{L2} and \eqref{ae3}.
The convergence \eqref{L8} and \eqref{L9}
are immediate results of \eqref{ae1} and \eqref{ae3}.
\end{proof}

\label{pfthm1}
Finally, all the convergence properties shown in the above two lemmata allow us to give the proof of our main result.\\[-10pt]
\begin{proof}[Proof of Theorem \ref{th 1}]
If $\gamma\ge 2$, then $f$ satisfies \eqref{f} also for any $\gamma\in(1,2)$ (possibly with a larger value of $r$); therefore the previous lemmata remain applicable.
Let $\varphi\in C_0^\infty(\overline{\Omega}\times[0,\infty))$ be an arbitrarily fixed nonnegative test function. According to the lower semicontinuity of $L^2$ norms with respect to weak convergence and the weak convergence of $\sqrt{φ} \nabla \ln(n_{\eps }+1)$ (by \eqref{L3}), we have
\begin{align}\label{t1}
\int_{0}^{\infty}\int_{\Omega}
|\nabla \ln\left(n+1\right)|^2\varphi\le
\liminf _{\varepsilon=\varepsilon_{j} \rightarrow 0}
\int_{0}^{\infty} \int_{\Omega} |\nabla \ln\left(n_{\varepsilon}+1\right)|^2 \varphi.
\end{align}
Due to Lemma~\ref{lem3.9} (and boundedness of $φ$),
$\left\{\left(n_{\varepsilon}+1\right)^{-1}n^2_{\varepsilon}\varphi\right\}_{\varepsilon \in(0,1)}$
is uniformly integrable over $\Omega \times(0, T)$.
This combined with the Vitali convergence theorem entails that
$(n_\varepsilon+1)^{-1} n^2_{\varepsilon}\varphi\rightarrow (n+1)^{-1} n^2\varphi$
in $L^{1}_{l o c}\big(\overline{\Omega}\times[0, \infty)\big)$,
and thus
\begin{align}\label{t2}
\lim _{\varepsilon=\varepsilon_{j} \rightarrow 0}
\varepsilon\int_{0}^{\infty} \int_{\Omega} (n_\varepsilon+1)^{-1} n^2_{\varepsilon} \varphi=0.
\end{align}
Hence, by integrating \eqref{a0}
with respect to the time-variable, we utilize \eqref{t1} and \eqref{t2} to get
\begin{align}\label{t3}
&\quad\int_{0}^{\infty}\int_{\Omega}|\nabla \ln\left(n+1\right)|^2\varphi\nonumber\\
&\le\liminf _{\varepsilon=\varepsilon_{j} \rightarrow 0}
\Bigg\{
\int_{0}^{\infty} \int_{\Omega} |\nabla \ln\left(n_\varepsilon+1\right)|^2\varphi -\varepsilon\int_{0}^{\infty} \int_{\Omega} (n_\varepsilon+1)^{-1} n^2_{\varepsilon} \varphi\Bigg\}\nonumber\\
&=\liminf _{\varepsilon=\varepsilon_{j} \rightarrow 0}\Bigg\{
-\int_{0}^{\infty} \int_{\Omega} \ln\left(n_\varepsilon+1\right)\varphi_{t}
-\int_{\Omega} \ln\left(n_{0\varepsilon}+1\right)\varphi(\cdot, 0)\nonumber\\
&\quad-\int_{0}^{\infty}\int_{\Omega}\ln\left(n_\varepsilon+1\right)\mathbf{u}_\varepsilon\cdot\nabla \varphi
 +\int_{0}^{\infty}\int_{\Omega} (n_\varepsilon+1)^{-2} n_\varepsilon
 \nabla n_\varepsilon\cdot\nabla c_\varepsilon\varphi
 \nonumber\\
 &\quad
+\int_{0}^{\infty}\int_{\Omega} (n_\varepsilon+1)^{-1}\nabla n_\varepsilon\cdot\nabla\varphi
 -\int_{0}^{\infty}\int_{\Omega} (n_\varepsilon+1)^{-1} n_\varepsilon
  \nabla c_\varepsilon\cdot\nabla\varphi
  \nonumber\\
  &\quad
 -\int_{0}^{\infty}\int_{\Omega} (n_\varepsilon+1)^{-1} f(n_\varepsilon)\varphi
 \Bigg\}.
\end{align}
On the right-hand side of \eqref{t3}, we can use the convergence properties previously derived and applying \eqref{L2}, \eqref{initial1}, \eqref{L7}, \eqref{L5}, \eqref{L3}, \eqref{L6} and \eqref{L4}, we obtain that
\begin{align}\label{t4}
\quad\int_{0}^{\infty} \int_{\Omega}\ln\left(n+1\right)\varphi_{t}&+
\int_{\Omega} \ln\left(n_0+1\right)\varphi(0)\nonumber\\
 &\le-\int_{0}^{\infty}\int_{\Omega}\ln\left(n+1\right)\mathbf{u}\cdot\nabla \varphi
 -\int_{0}^{\infty}\int_{\Omega}|\nabla\ln\left(n+1\right)|^2\varphi
 \nonumber\\
 &\quad+\int_{0}^{\infty}\int_{\Omega} (n+1)^{-1}\nabla n\cdot\nabla\varphi
+\int_{0}^{\infty}\int_{\Omega} (n+1)^{-2} n
 \nabla n\cdot\nabla c\varphi
 \nonumber\\
  &\quad-\int_{0}^{\infty}\int_{\Omega} (n+1)^{-1} n
  \nabla c\cdot\nabla\varphi
 -\int_{0}^{\infty}\int_{\Omega} (n+1)^{-1} f(n)\varphi.
\end{align}
The estimate
\begin{align}\label{t5}
\int_{\Omega} n(\cdot, t) \leq \int_{\Omega} n_{0}+\int_{0}^{t} \int_{\Omega} f(n) \quad \text { for }~ a.e.~t>0
\end{align}
results from \eqref{L1} by Fatou's lemma in the same way as detailed in
\cite[p.20]{W2019}.
In view of \eqref{t4} and \eqref{t5}, we can assert the function $n$ satisfies the conditions required in Definition \ref{def3}.
Based on Lemma \ref{lem4.1}, \eqref{L1}, \eqref{L8} and \eqref{L9}, it is easy to verify the functions $c$ and $\mathbf{u}$ satisfy the corresponding equations in the weak sense, as exhibited in Definition \ref{def1} and Definition \ref{def2}.
\end{proof}
\section{Eventual smoothness}

In this section, we focus on investigating the eventual regularity properties of the generalized solution $(n,c,\mathbf{u})$. The proof will be based on the eventual quasi-energy functional
\begin{equation*}
\io \nep\ln\nep+\f12\io|\nabla\cep|^2.
 \end{equation*}
So as to ensure that it is a quasi-energy functional, we will need smallness of the mass $\io \nep$. This is where the smallness of $r$ (or largeness of $\mu$) matters.
As soon as then, finally, (eventual) boundedness of $n$ in $L^p(\Omega)$ is achieved for large $p$, we can rely on standard procedures to prove regularity of the solution components in the corresponding space-time domain.

In order to state a value for $\mu_0$ in Theorem \ref{th 2},
we introduce $C_*$ (only depending on the domain $\Omega$) as the best constant in
the following Gagliardo-Nirenberg inequality:
\begin{align}\label{GN}
\|\varphi\|_{L^4(\Omega)}
\le C_*\big(\|\nabla \varphi\|^\frac{1}{2}_{L^2(\Omega)}\|\varphi\|^\frac{1}{2}_{L^2(\Omega)}
+\|\varphi\|_{L^2(\Omega)}\big) \qquad \forall \varphi\in W^{1,2}(\Omega).
\end{align}

Since most lemmata in Section~\ref{sec:estimates} included the condition $γ\in(1,2)$, which we want to avoid in the following, let us firstly collect some bounds that later proofs will rely on without this assumption.

\begin{lemma}
 Let $T>0$. If $γ\in(1,∞)$, there is $C>0$ such that for every $\eps \in(0,1)$,
 \begin{align}
  \int_0^T\io |∇\uep|^2 &\le C\label{bound:nabla-u-2},\\
 \int_0^T\io |∇\cep|^2  &\leq C.\label{bound:nabla-c-2}
 \end{align}
\end{lemma}
\begin{proof}
 In the same manner as during the proof of Theorem~\ref{th 1}, we remark that if $f$ satisfies \eqref{f} for some $γ\ge 2$, $f$ satisfies $\eqref{f}$ also for smaller values (in $(1,2)$) of $γ$, if $r$ and $\mu$ are adjusted as necessary. Therefore, the lemmata of Section~\ref{sec:estimates} are applicable and \eqref{bound:nabla-u-2} follows from \eqref{b0'}, \eqref{bound:nabla-c-2} from \eqref{u0'}.
\end{proof}

The next lemma is used to demonstrate that at large times, the $L^1(\Omega)$-norm of $n$ can be controlled by the system parameters.
 \begin{lemma}\label{lem5.1}
Then
\begin{align}\label{l1}
\limsup_{t\to\infty}\sup_{\eps \in(0,1)} \io \nep(x,t)dx \le |\Omega|\Big(\frac{r_+}{\mu}\Big)^{\frac{1}{\gamma-1}}.
\end{align}
\end{lemma}
\begin{proof}
Integrating \eqref{qapp}$_1$ over $\Omega$, we have
\begin{align}\label{l2}
\frac{d}{dt}\int_\Omega n_\varepsilon &\le r\int_\Omega n_\varepsilon -\mu\int_\Omega n_\varepsilon^\gamma \nonumber\\
&\le r\int_\Omega n_\varepsilon -\frac{\mu}{|\Omega|^{\gamma-1}}
\Big(\int_\Omega
n_\varepsilon \Big)^\gamma
\quad\text{in } (0,∞)~\text{for all}~\varepsilon\in(0,1).
\end{align}
We let $y\in C^0([0,\infty))\cap C^1((0,\infty))$ denote the solution of the initial value problem $y'=ry-\f{μ}{|\Omega|^{\gamma-1}}y^\gamma$, $y(0)=2\io n_0$ and note that $\io \nep\le y$ on $(0,\infty)$ for every $\eps \in(0,1)$ due to \eqref{l2} and \eqref{initial2}. Since $y(t)\to |\Omega|\left(\frac{r_+}{\mu}\right)^{\f1{\gamma-1}}$ as $t\to \infty$
, this shows \eqref{l1}. 
\end{proof}


In the treatment of the derivatives of $\io |\nabla \cep|^2$, the integral involving $\uep$ does not vanish like it did so often before. Therefore, further estimates for $\uep$ are required. The course of action is similar to that applied for $\cep$ in Lemma~\ref{lem3.7}. We again firstly introduce a suitable operator and collect a few basic results, analogous to Lemma~\ref{lem3.6}:


With $D(\mathcal{A}):=W^{2,2}\left(\Omega;\mathbb{R}^{2}\right) \cap W_{0}^{1,2}\left(\Omega ; \mathbb{R}^{2}\right) \cap L_{\sigma}^{2}(\Omega)$, we let $\mathcal{A}:=-\projection\Delta$ denote the realization of the Stokes operator on $D(\mathcal{A})$. Therein, $\projection$ stands for the Helmholtz projection from $L^2(\Omega;\mathbb{R}^2)$ to $L^{2}_\sigma(\Omega)$.
\begin{lemma}\label{lem5.2}
\textrm{(i)} $\mathcal{A}$ is a sectorial and positive self-adjoint operator, it possesses closed fractional powers  $\mathcal{A}^\alpha$ defined on $D(\mathcal{A}^{\alpha})$ with $\alpha\in \mathbb{R}$, where the norm is given by $D(\mathcal{A}^{\alpha}) :=\|\mathcal{A}^{\alpha}(\cdot)\|_{L^2(\Omega)}$.\\
\textrm{(ii)} For any $\alpha\in(0,1)$, the domain $D\left(\mathcal{A}^{\alpha}\right)$ is continuously embedded
in $L^2_{\sigma}(\Omega) \cap\left(W^{2\alpha,2}(\Omega;\mathbb{R}^2)\right)$.\\
\textrm{(iii)} For all $\alpha,\beta,\delta\in\mathbb{R}$ satisfying $\beta<\alpha<\delta$, there is $C>0$ such that
\begin{align*}
\|\mathcal{A}^\alpha\varphi\|_{L^2(\Omega)}\le C\|\mathcal{A}^\delta\varphi\|^{\frac{\alpha-\beta}{\delta-\beta}}_{L^2(\Omega)}
\|\mathcal{A}^\beta\varphi\|^{\frac{\delta-\alpha}{\delta-\beta}}_{L^2(\Omega)},
\end{align*}
for all $φ\in D(\mathcal{A}^{\delta})$.\\
\textrm{(iv)} Let $p>1$. If $\projection$ denotes the Helmholtz projection
from $L^p(\Omega;ℝ^2)$ to $L^p_{\sigma}(\Omega)$, then $\projection$ is a bounded linear operator.
\end{lemma}
\begin{proof}
 \cite[Proposition 4.1]{GM} and \cite[Theorem 14.1]{Fr}; and, for (iv), \cite[Thm. 1 and Thm. 2]{fujiwara_morimoto}.
\end{proof}

Finally relying on Lemma~\ref{ODI}, we then obtain the following regularity information on $\uep$:

\begin{lemma}\label{lem5.3}
For any $T>1$,
there exists $C>0$ such that
\begin{align}\label{h0}
&\int_{1}^{T}\int_{\Omega}|\mathcal{A}^{\frac{\delta+1}{2}}\mathbf{u}_\varepsilon(x, t)|^{2}dxdt \leq C\quad \text { for all } \varepsilon \in(0,1)
\end{align}
with $\delta=\min\{1/2,\gamma-1\}$.
\end{lemma}
\begin{proof}
By \eqref{bound:nabla-u-2}
, there exists $c_1>0$ such that
\begin{align}\label{h1}
&\int_{0}^{1}\int_{\Omega}|\nabla\mathbf{u}_\varepsilon(x, t)|^{2}dxdt \leq c_1\quad
\text { for all } \varepsilon \in(0,1).
\end{align}
For any fixed $\varepsilon\in(0,1)$, we can thus find $t_\varepsilon\in(0,1)$ satisfying
\begin{align}\label{h1'}
\|\mathcal{A}^{\frac{\delta}{2}}\mathbf{u}_{\varepsilon}(\cdot,t_\varepsilon) \|_{L^{2}(\Omega)}
 \leq c_{1}.
\end{align}
First, we rewrite \eqref{qapp}$_3$ and \eqref{qapp}$_4$ in the form that
\begin{align}\label{h2}
\mathbf{u}_{\varepsilon t}-\mathcal{A} \mathbf{u}_\varepsilon
        =-\kappa \projection(\mathbf{u}_{\varepsilon }\cdot\nabla)\mathbf{u}_{\varepsilon }+ \projection n_\varepsilon\nabla\phi, \quad \varepsilon\in(0,1),~t>0.
\end{align}
Multiplying $\mathcal{A}^{\delta}\mathbf{u}_\varepsilon$ to \eqref{h2}
 and using the self-adjointness of the operator $\mathcal{A}$, we have that
\begin{align}\label{h3}
&\quad\frac{1}{2}\frac{d}{d t} \int_{\Omega}|\mathcal{A}^{\frac{\delta}{2}}\mathbf{u}_\varepsilon|^{2}
+\int_{\Omega}|\mathcal{A}^{\frac{\delta+1}{2}}\mathbf{u}_\varepsilon|^{2} \nonumber\\
&=-\kappa\int_{\Omega}\projection(\mathbf{u}_{\varepsilon }\cdot\nabla)\mathbf{u}_{\varepsilon }\cdot \mathcal{A}^{\delta}\mathbf{u}_{\varepsilon }
+\int_{\Omega} \projection(n_\varepsilon\nabla\phi)\cdot \mathcal{A}^{\delta}\mathbf{u}_{\varepsilon }
\quad \text { in } (0,∞) \text { for all } \varepsilon \in(0,1).
\end{align}
With $p>2$ taken to satisfy $\frac{2}{p}+\delta<1$,
\begin{align}\label{h4}
-\int_{\Omega}\projection(\mathbf{u}_{\varepsilon }\cdot\nabla)\mathbf{u}_{\varepsilon }\cdot
 \mathcal{A}^{\delta}\mathbf{u}_{\varepsilon }
&\le \| \mathcal{A}^{\delta}\mathbf{u}_\varepsilon\|_{L^{\frac{2p}{p-2}}(\Omega)}
 \|\projection(\mathbf{u}_{\varepsilon }\cdot\nabla)\mathbf{u}_{\varepsilon}\|_{L^{\frac{2p}{p+2}}(\Omega)}\nonumber\\
 &\le c_2\| \mathcal{A}^{\delta}\mathbf{u}_\varepsilon\|_{L^{\frac{2p}{p-2}}(\Omega)}
\|(\mathbf{u}_{\varepsilon }\cdot\nabla)\mathbf{u}_{\varepsilon}\|_{L^{\frac{2p}{p+2}}(\Omega)}
  \\
 &\le c_2\| \mathcal{A}^{\delta}\mathbf{u}_\varepsilon\|_{L^{\frac{2p}{p-2}}(\Omega)}
 \| \nabla \mathbf{u}_\varepsilon\|_{L^{2}(\Omega)}
  \|\mathbf{u}_\varepsilon\|_{L^{p}(\Omega)}\nonumber
  \quad \text { in } (0,∞)\text { for all } \varepsilon \in(0,1)
\end{align}
with $c_2=c_2(\frac{2p}{p-2})>0$ originating from Lemma~\ref{lem5.2}(iv).
By using the Sobolev inequality \cite[Theorem 6.7]{NPV} and Lemma \ref{lem5.2}(ii) and (iii), we obtain
\begin{align}\label{h5}
\| \mathcal{A}^{\delta}\mathbf{u}_\varepsilon\|_{L^{\frac{2p}{p-2}}(\Omega)}
&\le c_3\|\mathcal{A}^{\delta}\mathbf{u}_\varepsilon\|_{W^{\frac{2}{p}, 2}(\Omega)}\nonumber\\
&\le c_4\| \mathcal{A}^{\delta+\frac{1}{p}}\mathbf{u}_\varepsilon\|_{L^{2}(\Omega)}\nonumber\\
&\le c_5\|
\mathcal{A}^{\frac{\delta+1}{2}}\mathbf{u}_\varepsilon\|^{\delta+\frac{2}{p}}_{L^{2}(\Omega)}
\| \mathcal{A}^{\frac{\delta}{2}}\mathbf{u}_\varepsilon\|^{1-\delta-\frac{2}{p}}_{L^{2}(\Omega)}
\quad \text { in } (0,∞) \text { for all } \varepsilon \in(0,1)
\end{align}
with positive constants $c_3,c_4,c_5$ depending on $\sigma,\gamma,p$ and $\Omega$.
Similarly, still by the Sobolev imbedding inequality and Lemma \ref{lem5.2}(ii), (iii),
we can find $c_6,c_7,c_8>0$ such that
\begin{align}\label{h6}
\|\mathbf{u}_\varepsilon\|_{L^{p}(\Omega)}
&\le c_6\|\mathbf{u}_\varepsilon\|_{W^{1-\frac{2}{p}, 2}(\Omega)}\nonumber\\
&\le c_7\| \mathcal{A}^{\frac{1}{2}-\frac{1}{p}}\mathbf{u}_\varepsilon\|_{L^{2}(\Omega)}\nonumber\\
&\le c_8\|
\mathcal{A}^{\frac{\delta+1}{2}}\mathbf{u}_\varepsilon\|^{1-\delta-\frac{2}{p}}_{L^{2}(\Omega)}
\|\mathcal{A}^{\frac{\delta}{2}}\mathbf{u}_\varepsilon\|^{\delta+\frac{2}{p}}_{L^{2}(\Omega)}
\quad \text { in } (0,∞) \text { for all } \varepsilon \in(0,1).
\end{align}
A substitution of \eqref{h5} and \eqref{h6} into \eqref{h4} and Young's inequality give that with $c_9=\kappa^2c_2^2c_5^2c_8^2$
\begin{align}\label{h7}
-\kappa\int_{\Omega}\projection(\mathbf{u}_{\varepsilon }\cdot\nabla)\mathbf{u}_{\varepsilon }
 \cdot\mathcal{A}^{\delta}\mathbf{u}_{\varepsilon }
 &\le \kappa c_2c_5c_8\| \mathcal{A}^{\frac{\delta+1}{2}}\mathbf{u}_\varepsilon\|_{L^{2}(\Omega)}
 \| \nabla \mathbf{u}_\varepsilon\|_{L^{2}(\Omega)}
  \| \mathcal{A}^{\frac{\delta}{2}}\mathbf{u}_\varepsilon\|_{L^{2}(\Omega)}\nonumber\\
  &\le \frac{1}{4}\| \mathcal{A}^{\frac{\delta+1}{2}}\mathbf{u}_\varepsilon\|^2_{L^{2}(\Omega)}
 +c_9 \| \nabla \mathbf{u}_\varepsilon\|^2_{L^{2}(\Omega)}
  \| \mathcal{A}^{\frac{\delta}{2}}\mathbf{u}_\varepsilon\|^2_{L^{2}(\Omega)}
\end{align}
in $(0,∞)$ for all $\varepsilon \in(0,1)$.

We deal with the remaining term $\int_{\Omega}\projection(n_\varepsilon\nabla\phi)\cdot \mathcal{A}^{\delta}\mathbf{u}_{\varepsilon }$ for the cases $\gamma\in(1,2)$
and $\gamma\ge2$ separately, beginning with $γ\in(1,2)$:

Since $\delta\le\gamma-1$, the Sobolev imbedding inequality and Lemma \ref{lem5.2}(ii) ensure the existence of $c_{10},c_{11}$ satisfying
\begin{align}\label{h8}
\|\mathcal{A}^{\delta}\mathbf{u}_\varepsilon\|_{L^{\frac{2}{\gamma-1}}(\Omega)}
&\le c_{10}\|\mathcal{A}^{\delta}\mathbf{u}_\varepsilon\|_{W^{1-\delta,2}(\Omega)}\nonumber\\
&\le
c_{11}\|\mathcal{A}^{\frac{\delta+1}{2}}\mathbf{u}_\varepsilon\|_{L^{2}(\Omega)}
\quad \text { for all } t>0\text { and } \varepsilon \in(0,1).
\end{align}
It follows by the H\"{o}lder inequality and Young's inequality that in $(0,∞)$ with some $c_{12}>0$ taken from an application of Lemma~\ref{lem5.2}(iv) to $L^{\frac{2}{3-γ}}(\Omega)$,
\begin{align*}
\int_{\Omega}\projection(n_\varepsilon\nabla\phi)\cdot \mathcal{A}^{\delta}\mathbf{u}_{\varepsilon }
 &\le
 \| \projection(n_\varepsilon\nabla\phi)\|_{L^{\frac{2}{3-\gamma}}(\Omega)}
 \|\mathcal{A}^{\delta}\mathbf{u}_\varepsilon\|_{L^{\frac{2}{\gamma-1}}(\Omega)}\nonumber\\
   &\le c_{12}\|\nabla\phi\|_{L^{\infty}(\Omega)}
   \| n_\varepsilon\|_{L^{\frac{2}{3-\gamma}}(\Omega)}
 \|\mathcal{A}^{\delta}\mathbf{u}_\varepsilon\|_{L^{\frac{2}{\gamma-1}}(\Omega)}\nonumber\\
   &\le
   c_{11}^2 c_{12}^2\|\nabla\phi\|^2_{L^{\infty}(\Omega)}\| n_\varepsilon\|^2_{L^{\frac{2}{3-\gamma}}(\Omega)}+
 \frac{1}{4c_{11}^2}
 \|\mathcal{A}^{\delta}\mathbf{u}_\varepsilon\|^2_{L^{\frac{2}{\gamma-1}}(\Omega)} \text { for all } \varepsilon \in(0,1),
\end{align*}
which in conjunction with \eqref{h8} tells us that
\begin{align}\label{h9}
\int_{\Omega}\projection(n_\varepsilon\nabla\phi) \cdot\mathcal{A}^{\delta}\mathbf{u}_{\varepsilon }
\le
   c_{13}\| n_\varepsilon\|^2_{L^{\frac{2}{3-\gamma}}(\Omega)}+
 \frac{1}{4}\|\mathcal{A}^{\frac{\delta+1}{2}}\mathbf{u}_\varepsilon\|^2_{L^{2}(\Omega)}
 \quad \text { in } (0,∞) \text { for all } \varepsilon \in(0,1)
\end{align}
with $c_{13}=c_{11}^2c_{12}^2\norm[\Lom{\infty}]{\nabla \phi}>0$.\\
In the case $\gamma\ge2$, we instead introduce $c_{14}>0$ such that
\begin{equation*}
\|\mathcal{A}^{\delta}\mathbf{u}_\varepsilon\|_{L^{2}(\Omega)}\le c_{14}\|\mathcal{A}^{\frac{\delta+1}{2}}\mathbf{u}_\varepsilon\|_{L^{2}(\Omega)}\quad\text{in } (0,∞) \text{ for all } \eps \in(0,1)
\end{equation*}
and, with $c_{15}=c_{14}^2\norm[\Lom\infty]{\nabla\phi}^2$, estimate
 \begin{align}\label{h9b}
\int_{\Omega}\projection(n_\varepsilon\nabla\phi) \cdot\mathcal{A}^{\delta}\mathbf{u}_{\varepsilon }
 &\le
 \| \projection(n_\varepsilon\nabla\phi)\|_{L^{2}(\Omega)}
 \|\mathcal{A}^{\delta}\mathbf{u}_\varepsilon\|_{L^{2}(\Omega)}\nonumber\\
   &\le \|\nabla\phi\|_{L^{\infty}(\Omega)}
   \| n_\varepsilon\|_{L^{2}(\Omega)}
 \|\mathcal{A}^{\delta}\mathbf{u}_\varepsilon\|_{L^{2}(\Omega)}\nonumber\\
   &\le
   c_{14}^2 \|\nabla\phi\|^2_{L^{\infty}(\Omega)}\| n_\varepsilon\|^2_{L^{2}(\Omega)}+
 \frac{1}{4c_{14}^2}
 \|\mathcal{A}^{\delta}\mathbf{u}_\varepsilon\|^2_{L^{2}(\Omega)}
 \nonumber\\
 &\le c_{15}\| n_\varepsilon\|^2_{L^{2}(\Omega)}+
 \frac{1}{4}\|\mathcal{A}^{\frac{\delta+1}{2}}\mathbf{u}_\varepsilon\|^2_{L^{2}(\Omega)}.
\end{align}
In both cases $γ\in(1,2)$ and $γ\ge 2$, we therefore can conclude from \eqref{h3}, \eqref{h7} and either \eqref{h9} or \eqref{h9b} that
\begin{align}\label{h10}
&\quad\frac{d}{d t} \int_{\Omega}|\mathcal{A}^{\frac{\delta}{2}}\mathbf{u}_\varepsilon(x,t)|^{2} dx
+\int_{\Omega}|\mathcal{A}^{\frac{\delta+1}{2}}\mathbf{u}_\varepsilon(x,t)|^{2} dx \le a(t)
  \| \mathcal{A}^{\frac{\delta}{2}}\mathbf{u}_\varepsilon(\cdot,t)\|^2_{L^{2}(\Omega)}
  + b(t)
\end{align}
for all $t>0$ and $\eps \in(0,1)$, where
\begin{align*}
a_{\eps }(t) =
2c_9\| \nabla \mathbf{u}_\varepsilon(\cdot,t)\|^2_{L^{2}(\Omega)}
\quad \text{and}\quad
b_{\eps }(t)=\begin{cases}2c_{13}\| n_\varepsilon(\cdot,t)\|^2_{L^{\frac{2}{3-\gamma}}(\Omega)}, &γ\in(1,2),\\
2c_{15}\norm[\Lom2]{\nep(\cdot,t)}^2,& \gamma\ge 2.\end{cases}
\end{align*}
We observe that the finiteness of $\sup_{\eps \in(0,1)}\int_0^T a_{\eps }(t)dt$ is covered by \eqref{bound:nabla-u-2}
and that of $\sup_{\eps \in(0,1)}\int_0^T b_{\eps }(t)dt$ by either Corollary~\ref{cor1} or \eqref{c3}, so that Lemma~\ref{ODI} is applicable and, if combined with \eqref{h1'}, ensures \eqref{h0}.
\end{proof}

With this better regularity of $\uep$ ensured, we are able to deal with the integrals involving convective terms arising during an investigation of the temporal evolution of
\[
 \io \nep\ln \nep+ \f12\io |∇\cep|^2.
\]

\begin{lemma}\label{lem5.4}
Let $C_*>0$ be given in \eqref{GN}. 
If $\mu> r_+\big(16C_*^4|\Omega|/3\big)^{\gamma-1}$, then there is $t_1>1$ such that for any $T>t_1$,
we can find $C>0$ ensuring
\begin{align}\label{m0}
&\int_{\Omega}|\nabla c_\varepsilon(\cdot, t)|^{2}
\leq C\quad\text { for all } t\in(t_1,T)\text { and } \varepsilon \in(0,1)
\end{align}
and
\begin{align}\label{m0'}
&\int_{t_1}^{T}\int_\Omega n_{\varepsilon}^{2}(x,t)dxdt
+\int_{t_1}^{T}\int_{\Omega}|\Delta c_\varepsilon(x, t)|^{2}dxdt \leq C\quad \text { for all } \varepsilon \in(0,1).
\end{align}
\end{lemma}
\begin{proof}
By utilizing \eqref{qapp}$_1$ and integrating by parts, we have
\begin{align}\label{m3}
\frac{d}{d t} \int_{\Omega}n_\varepsilon\ln n_\varepsilon
&=\int_{\Omega}n_{\varepsilon t}\ln n_\varepsilon
+\int_{\Omega}n_{\varepsilon t} \nonumber\\
&=-\int_{\Omega}\frac{|\nabla n_\varepsilon|^2}{n_\varepsilon}
-\int_{\Omega}\mathbf{u}_\varepsilon \cdot\nabla \left(\int_{0}^{n_\varepsilon}\ln \xi d\xi\right)
-\int_{\Omega}\nabla n_\varepsilon \cdot \nabla c_\varepsilon
\nonumber\\
&\quad+r\int_{\Omega}n_\varepsilon (\ln n_\varepsilon+1)
-\mu\int_{\Omega}n^\gamma_\varepsilon(\ln n_\varepsilon+1)
-\varepsilon\int_{\Omega}n^2_\varepsilon(\ln n_\varepsilon+1)
\nonumber\\
&\le-4\int_{\Omega}|\nabla \sqrt{n_\varepsilon}|^2
+\int_{\Omega}n_\varepsilon \Delta c_\varepsilon +c_2
\quad \text { in } (0,\infty)\text { for all } \varepsilon \in(0,1),
\end{align}
where $c_2=|\Omega|\sup_{s>0} (rs-\mu s^\gamma)(1+\ln s) - |\Omega|\inf_{s>0} s^2\ln s$, and
the second term on
the right-hand side can be estimated by Young's inequality according to
\begin{align}\label{m4}
\int_{\Omega}n_\varepsilon \Delta c_\varepsilon \le \int_{\Omega}n^2_\varepsilon
+\frac{1}{4}\int_{\Omega} |\Delta c_\varepsilon|^2
\quad \text { in } (0,\infty)\text { for all } \varepsilon \in(0,1).
\end{align}
Thus, it can be obtained by \eqref{m3} that
\begin{align}\label{m5}
\frac{d}{d t} \int_{\Omega}n_\varepsilon\ln n_\varepsilon
+4\int_{\Omega}|\nabla \sqrt{n_\varepsilon}|^2
&\le\frac{1}{4}\int_{\Omega} |\Delta c_\varepsilon|^2+ \int_{\Omega}n^2_\varepsilon+c_2
\quad \text { in } (0,\infty)\text { for all } \varepsilon \in(0,1).
\end{align}
 We multiply \eqref{qapp}$_2$ by $\Delta c_\varepsilon$ and use Young's inequality to get
\begin{align}\label{m6}
\frac{1}{2}\frac{d}{d t} \int_{\Omega}|\nabla c_\varepsilon|^2
&+\int_{\Omega}|\Delta c_\varepsilon|^2 +\int_{\Omega}|\nabla c_\varepsilon|^2 \nonumber\\
&=\int_{\Omega} \mathbf{u}_\varepsilon\cdot\nabla c_\varepsilon \Delta c_\varepsilon
+\int_{\Omega}\frac{n_\varepsilon}{1+\varepsilon n_\varepsilon}\Delta c_\varepsilon \nonumber\\
&\le \frac{1}{2}\int_{\Omega}|\Delta c_\varepsilon|^2
+\int_{\Omega}|\mathbf{u}_\varepsilon\cdot\nabla c_\varepsilon|^2
+\int_{\Omega}n^2_\varepsilon
\quad \text { in } (0,\infty)\text { for all } \varepsilon \in(0,1),
\end{align}
where the Sobolev imbedding inequality tells us that for $\delta=\min\{\frac12,γ-1\}$ as in Lemma \ref{lem5.3} and with some $c_3=c_3(\gamma,\Omega)>0$,
\begin{align*}
\int_{\Omega}|\mathbf{u}_\varepsilon\cdot\nabla c_\varepsilon|^2
&\le \|\mathbf{u}_\varepsilon\|^2_{L^{\infty}(\Omega)}
\|\nabla c_\varepsilon\|^2_{L^{2}(\Omega)}\nonumber\\
&\le c_3\|\mathcal{A}^{\frac{\delta+1}{2}} \mathbf{u}_\varepsilon\|^2_{L^{2}(\Omega)}
\|\nabla c_\varepsilon\|^2_{L^{2}(\Omega)}
\quad \text { in } (0,\infty)\text { for all } \varepsilon \in(0,1).
\end{align*}
Thus, we have
\begin{align}\label{m7}
&\quad\frac{1}{2}\frac{d}{d t} \int_{\Omega}|\nabla c_\varepsilon|^2
+\frac{1}{2}\int_{\Omega}|\Delta c_\varepsilon|^2
+\int_{\Omega}|\nabla c_\varepsilon|^2 \nonumber\\
&
\le c_3\|\mathcal{A}^{\frac{\delta+1}{2}} \mathbf{u}_\varepsilon\|^2_{L^{2}(\Omega)}
\|\nabla c_\varepsilon\|^2_{L^{2}(\Omega)}
+\int_{\Omega}n_\varepsilon^2
\quad \text { in } (0,\infty)\text { for all } \varepsilon \in(0,1).
\end{align}
A combination of \eqref{m5} and \eqref{m7} yields that
\begin{align}\label{b7}
\frac{d}{d t} \Big(\int_{\Omega}n_\varepsilon\ln n_\varepsilon
&+\frac{1}{2}\int_{\Omega}|\nabla c_\varepsilon|^2 \Big)
+4\int_{\Omega}|\nabla \sqrt{n_\varepsilon}|^2
+\frac{1}{4}\int_{\Omega}|\Delta c_\varepsilon|^2 \nonumber\\
&\le c_3\|\mathcal{A}^{\frac{\delta+1}{2}} \mathbf{u}_\varepsilon\|^2_{L^{2}(\Omega)}
\|\nabla c_\varepsilon\|^2_{L^{2}(\Omega)}
+2\int_{\Omega}n_\varepsilon^2 +c_2
\quad \text { in } (0,\infty)\text { for all } \varepsilon \in(0,1).
\end{align}
An application of the Gagliardo-Nirenberg inequality \eqref{GN} shows that
\begin{align}\label{b8}
\|n_\varepsilon\|^2_{L^2(\Omega)}&=\|\sqrt{n_\varepsilon}\|^4_{L^4(\Omega)}\nonumber\\
&
\le 8C_*^4\|\nabla \sqrt{n_\varepsilon}\|^2_{L^2(\Omega)}\|\sqrt{n_\varepsilon}\|^2_{L^2(\Omega)}
+8C_*^4\|\sqrt{n_\varepsilon}\|^2_{L^2(\Omega)}
\quad \text { in } (0,\infty)\text { for all } \varepsilon \in(0,1).
\end{align}
Therefore,
\begin{align}\label{better-b7}
&\quad\frac{d}{d t} \Big(\int_{\Omega}n_\varepsilon\ln n_\varepsilon+\frac{1}{2}\int_{\Omega}|\nabla c_\varepsilon|^2 \Big)
+\left(4-16C_*^4\io \nep \right)\int_{\Omega}|\nabla \sqrt{n_\varepsilon}|^2
+\frac{1}{4}\int_{\Omega}|\Delta c_\varepsilon|^2 \nonumber\\
&\le c_3\|\mathcal{A}^{\frac{\delta+1}{2}} \mathbf{u}_\varepsilon\|^2_{L^{2}(\Omega)}
\|\nabla c_\varepsilon\|^2_{L^{2}(\Omega)}
+c_2+16C_*^4\io\nep\qquad \text{in } (0,\infty).
\end{align}
We now employ Lemma \ref{lem5.1} to find $t_1>2$ such that $\io \nep(\cdot,t) < \f3{16C_*^4}$ for every $t>t_1-1$ and every $\eps \in(0,1)$, which is possible since the assumption $\mu>r_+\big(16C_*^4|\Omega|/3\big)^{\gamma-1}$ ensures that $|\Omega|(\f{r_+}{\mu})^{\f1{\gamma-1}}<\f{3}{16C_*^4}$.
%
We thus deduce from \eqref{better-b7} that for $c_4=c_2+3$
\begin{align}\label{b9-first}
&\quad\frac{d}{d t} \Big(\int_{\Omega}n_\varepsilon\ln n_\varepsilon
+\frac{1}{2}\int_{\Omega}|\nabla c_\varepsilon|^2 \Big)+\int_{\Omega}|\nabla \sqrt{n_\varepsilon}|^2
+\frac{1}{4}\int_{\Omega}|\Delta c_\varepsilon|^2 \nonumber\\
&\le c_3\|\mathcal{A}^{\frac{\delta+1}{2}} \mathbf{u}_\varepsilon\|^2_{L^{2}(\Omega)}
\|\nabla c_\varepsilon\|^2_{L^{2}(\Omega)}+c_4
\quad \text { in } (t_1-1, T)\text { for all } \varepsilon \in(0,1).
\end{align}
Based on \eqref{c3} and \eqref{bound:nabla-c-2}, there is $c_{5}>0$ such that
\begin{align}\label{m1}
\int_{t_1-1}^{t_1}\int_\Omega n_\varepsilon\ln n_\varepsilon
+\int_{t_1-1}^{t_1}\int_\Omega |\nabla c_\varepsilon|^2
\le c_{5} \qquad \text{for every } \eps \in(0,1).
\end{align}
Thus, for any fixed $\varepsilon\in(0,1)$, there is $t_\varepsilon\in(t_1-1,t_1)$ such that
\begin{align}\label{m2}
\int_\Omega (n_\varepsilon\ln n_\varepsilon)(x,t_\varepsilon) dx
+\int_\Omega |\nabla c_\varepsilon(x,t_\varepsilon)|^2 dx
\le c_{5}.
\end{align}
We observe that $\f1{e}+z \ln z\ge 0$ for every $z>0$, and recalling \eqref{m2} and \eqref{h0} (for the latter additionally relying on $t_{\eps }>t_1-1>1$), we utilize the boundedness
result given in Lemma \ref{ODI} to find
a constant $c_6>0$ ensuring
\begin{align}\label{b9-second}
\int_{\Omega}|\nabla c_\varepsilon(\cdot, t)|^{2}
\leq c_{6}\quad\text { for all } t\in(t_\varepsilon,T)\text { and } \varepsilon \in(0,1)
\end{align}
and
\begin{align}\label{b10}
\int_{t_\varepsilon}^T\int_{\Omega}|\Delta c_\varepsilon|^2
+\int_{t_\varepsilon}^T\int_{\Omega}|\nabla \sqrt{n_\varepsilon}|^2
\le c_{6}
\quad\text { for all } \varepsilon \in(0,1).
\end{align}
This together with \eqref{b8} implies the existence of $c_{7}>0$ satisfying
\begin{align*}
\int_{t_1}^T\int_{\Omega}n_\varepsilon^2 \le
\int_{t_\varepsilon}^T\int_{\Omega}n_\varepsilon^2
\le c_{7}\quad\text { for all } \varepsilon \in(0,1),
\end{align*}
as desired.
\end{proof}

Motivated by \cite[Lemma 3.8]{TW2016}, the two following results are used to establish higher integrability of $n_\varepsilon$.

\begin{lemma}\label{lem5.5}
Let $C_*>0$ be as specified in \eqref{GN}.
Assume that $\mu> r_+\big(16C_*^4|\Omega|/3\big)^{\gamma-1}$.
Let $t_1>1$ be as specified in Lemma \ref{lem5.4}.
Let $\tilde{t}>t_1$ and $T>\tilde{t}$. If there exist $p\ge2$ and $L>0$ such that
\begin{align}\label{z0}
&\int_{\tilde{t}-1}^{\tilde{t}}\int_{\Omega}n^{p}_\varepsilon(x, t)dxdt
\leq L\quad\text { for all }\varepsilon \in(0,1),
\end{align}
then we can find $C>0$ satisfying
\begin{align}\label{z0'}
&\int_{\Omega}n^{p}_\varepsilon(\cdot, t)
\leq C
\quad\text { for all } t\in(\tilde{t},T)\text { and } \varepsilon \in(0,1)
\end{align}
and
\begin{align}\label{z0''}
&\int_{\tilde{t}}^{T}\int_{\Omega}n^{p+\gamma-1}_\varepsilon(x, t)dxdt
\leq C\quad\text { for all } \varepsilon \in(0,1).
\end{align}
\end{lemma}
\begin{proof}
Due to \eqref{z0}, for any fixed $\varepsilon\in(0,1)$,
we can find $\tilde{t}_\varepsilon\in (\tilde{t}-1,\tilde{t})$
satisfying
\begin{align}\label{z1}
\int_\Omega n^p_\varepsilon(x,\tilde{t}_\varepsilon) dx
\le L\quad\text { for all } \varepsilon \in(0,1).
\end{align}
We multiply \eqref{qapp}$_1$ by $n_\varepsilon^{p-1}$ and
integrate by parts to see that
\begin{align}\label{z2}
\frac{1}{p}\frac{d}{d t} \int_{\Omega}n^p_\varepsilon
&=-(p-1)\int_{\Omega} n_\varepsilon^{p-2}|\nabla n_\varepsilon|^2
-\frac{1}{p}\int_{\Omega}\mathbf{u}_\varepsilon \cdot\nabla n^{p}_\varepsilon
-\frac{p-1}{p}\int_{\Omega}\nabla n^p_\varepsilon \cdot \nabla c_\varepsilon
\nonumber\\
&\quad+r\int_{\Omega}n^p_\varepsilon
-\mu\int_{\Omega}n^{p+\gamma-1}_\varepsilon -\varepsilon\int_{\Omega}n^{p+1}_\varepsilon
\nonumber\\
&\le-\frac{4(p-1)}{p^2}\int_{\Omega}|\nabla n^{\frac{p}{2}}_\varepsilon|^2
+\frac{p-1}{p}\int_{\Omega}n^p_\varepsilon \cdot \Delta c_\varepsilon
+r\int_{\Omega}n^p_\varepsilon
-\mu\int_{\Omega}n^{p+\gamma-1}_\varepsilon
\end{align}
in $(0,\infty)$ for any $\varepsilon \in(0,1)$;
here we used the fact $\nabla\cdot \mathbf{u}_\varepsilon=0$ in $\Omega\times(0,\infty)$.
It follows by H\"{o}lder's inequality that
\begin{align}\label{z3}
\int_{\Omega}n^p_\varepsilon \cdot \Delta c_\varepsilon
\le \Big(\int_{\Omega}n^{2p}_\varepsilon\Big)^{\frac{1}{2}}
\Big(\int_{\Omega} |\Delta c_\varepsilon|^2\Big)^{\frac{1}{2}}
\quad\text {in } (0,\infty)\text { for all } \varepsilon \in(0,1).
\end{align}
The Gagliardo-Nirenberg inequality and \eqref{c2} allow us to find $c_1,c_2 > 0$ such that
\begin{align}\label{z4}
\left(\int_{\Omega} n^{2 p}_\varepsilon\right)^{\frac{1}{2}}
&=\big\|n^{\frac{p}{2}}_\varepsilon\big\|_{L^{4}(\Omega)}^{2}\nonumber\\
& \leq c_{1}\big\|\nabla
n_\varepsilon^{\frac{p}{2}}\big\|_{L^{2}(\Omega)}
\big\|n_\varepsilon^{\frac{p}{2}}\big\|_{L^{2}(\Omega)}
+c_{1}\big\|n_\varepsilon^{\frac{p}{2}}\big\|_{L^{\frac{2}{p}}(\Omega)}^{2} \nonumber\\
& \leq c_{1}\big\|\nabla n_\varepsilon^{\frac{p}{2}}\big\|_{L^{2}(\Omega)}
\big\|n_\varepsilon^{\frac{p}{2}}\big\|_{L^{2}(\Omega)}
+c_{2}\quad\text {in } (0,\infty)\text { for all } \varepsilon \in(0,1).
\end{align}
In view of \eqref{z3} and \eqref{z4}, we infer from Young's inequality that
\begin{align}\label{z5}
\int_{\Omega}n^p_\varepsilon \cdot \Delta c_\varepsilon
&\le \frac{2(p-1)}{p^2}\int_{\Omega}|\nabla n^{\frac{p}{2}}_\varepsilon|^2
+c_3\|\Delta c_\varepsilon\|^2_{L^2(\Omega)}\int_{\Omega}n^p_\varepsilon  \nonumber\\
& \quad+c_3\int_{\Omega} |\Delta c_\varepsilon|^2 +c_3
\quad\text {in } (0,\infty)\text { for all } \varepsilon \in(0,1)
\end{align}
with $c_3=\max\{\frac{p^2c_1^2}{8(p-1)},\frac{c_2}2\}$. It is easy to verify that
\begin{align}\label{z5'}
r\int_{\Omega}n^p_\varepsilon
\le \frac{\mu}{2}\int_{\Omega}n^{p+\gamma-1}_\varepsilon +c_4
\quad\text {in } (0,\infty)\text { for all } \varepsilon \in(0,1)
\end{align}
 with $c_4=|\Omega|\sup_{s>0} (rs^p-\frac{\mu}2 s^{p+\gamma-1})$. 
 We derive from \eqref{z2}, \eqref{z5} and \eqref{z5'} that
\begin{align}\label{z6}
&\quad \frac{1}{p}\frac{d}{d t} \int_{\Omega}n^p_\varepsilon
+\frac{\mu}{2}\int_{\Omega}n^{p+\gamma-1}_\varepsilon \nonumber\\
&\le c_3 \|\Delta c_\varepsilon\|^2_{L^2(\Omega)}\int_{\Omega}n^p_\varepsilon +c_3\int_{\Omega} |\Delta c_\varepsilon|^2+c_3+c_4
\quad\text {in } (0,\infty)\text { for all } \varepsilon \in(0,1).
\end{align}
Based on \eqref{z1} and \eqref{m0'}, we apply Lemma \ref{ODI} to \eqref{z6} to find $c_5>0$ such that
\begin{align}\label{z7}
\sup_{t\in[\tilde{t}_\varepsilon,T)}\int_{\Omega}n^p_\varepsilon(x,t) dx
\le c_5
\end{align}
and
\begin{align}\label{z8}
\int_{\tilde{t}_\varepsilon}^T\int_{\Omega}n^{p+\gamma-1}_\varepsilon (x,t)dxdt
\le c_5
\end{align}
hold for all $\eps\in(0,1)$. Hence, the proof is complete,
because $(\tilde{t},T)\subset (\tilde{t}_{\eps },T)$ for any $\eps \in(0,1)$.
\end{proof}

\begin{corollary}\label{cor3}
Let $C_*>0$ be as specified in \eqref{GN} and assume that
 $\mu> r_+\big(16C_*^4|\Omega|/3\big)^{\gamma-1}$.
 Let $t_1>1$ be as provided by Lemma \ref{lem5.4}.
Then for each $p\ge2$, there exists $T_p\ge t_1+3$ with the property that for any
$T>T_p$, we can find $C>0$ satisfying
\begin{align*}
&\int_{\Omega}n^{p}_\varepsilon(\cdot, t)
\leq C\quad\text { for all } t\in(T_p,T)\text { and } \varepsilon \in(0,1).
\end{align*}
\end{corollary}
\begin{proof}
Taking \eqref{m0'} as a starting estimate, we can utilize Lemma \ref{lem5.5} to perform an iteration
procedure which results in the desired outcome.
\end{proof}

With these preparations, we are in the position to prove Theorem \ref{th 2}.
\begin{proof}[Proof of Theorem \ref{th 2}]
Let $\mu_0:=\big(16C_*^4|\Omega|/3\big)^{\gamma-1}$ with the pure constant $C_*>0$ defined in \eqref{GN}.
Let the assumptions in Theorem \ref{th 2} hold and $t_1>0$ be the time point specified in Lemma \ref{lem5.4}.
Since $\mu>\mu_0r_+$, Corollary \ref{cor3} guarantees the existence of $\tilde{T}\ge t_1+3$ such that for any $T>\tilde{T}$+2,
there exists $c_1>0$ satisfying
\begin{align}\label{t21}
\|n_\varepsilon(\cdot, t)\|_{L^{4\gamma}(\Omega)}
\leq c_1\quad\text { for all } t\in(\tilde{T},T)\text { and } \varepsilon \in(0,1).
\end{align}
By utilizing smoothing properties of the semigroups $\{{\e}^{t\mathcal{A}}\}_{t\ge0}$ and $\{{\e}^{t\Delta}\}_{t\ge0}$,
we can see that
\begin{align}\label{t22}
\|\mathbf{u}_\varepsilon(\cdot, t)\|_{L^{\infty}(\overline{\Omega})}
\leq c_1\quad\text { for all } t\in(\tilde{T}+1,T)\text { and } \varepsilon \in(0,1)
\end{align}
and
\begin{align}\label{t23}
\|\nabla c_\varepsilon(\cdot, t)\|_{L^{\infty}(\Omega)}
\leq c_2\quad\text { for all } t\in(\tilde{T}+1,T)\text { and } \varepsilon \in(0,1),
\end{align}
where the detailed proofs for \eqref{t22} and \eqref{t23} can be found
in \cite[Lemma 3.11 and 3.12]{TW2016}.
With the aid of \eqref{t21}-\eqref{t23}, we can adopt the arguments used in \cite[Lemma 3.13]{TW2016} to
deduce that
\begin{align}\label{t24}
\|n_\varepsilon(\cdot, t)\|_{C^{\theta}(\overline{\Omega})}
\leq c_3\quad\text { for all } t\in(\tilde{T}+2,T)\text { and } \varepsilon \in(0,1).
\end{align}
with $\theta\in(0,1)$.
By involving the regularity theory for parabolic equations and the Stokes
semigroup \cite{GS,HP}, it can be obtained that $n_\varepsilon$, $c_\varepsilon$ and $u_\varepsilon$
enjoy the uniform Schauder estimates as below,
\begin{align*}
\|n_\varepsilon\|_{C^{2+\lambda,1+\frac{\lambda}{2}}(\overline{\Omega}\times[\tilde{T}+2,T])}
+\|c_\varepsilon\|_{C^{2+\lambda,1+\frac{\lambda}{2}}(\overline{\Omega}\times[\tilde{T}+2,T])}
+\|\mathbf{u}_\varepsilon\|_{C^{2+\lambda,1+\frac{\lambda}{2}}(\overline{\Omega}\times[\tilde{T}+2,T])}
\leq c_3\quad\text { for all } \varepsilon \in(0,1)
\end{align*}
with some $\lambda\in(0,1)$ and $c_3=c_3(T)>0$, here we refer the reader to \cite[Section 3]{L2016}
for more details. According to the Arzel\`a-Ascoli theorem, we can claim that $(n,c,\mathbf{u})$ together with some $P$
classically solves \eqref{q1} on the time interval $(\tilde{T}+2,T)$ and
satisfies that
\begin{align*}
\|n\|_{C^{2,1}(\overline{\Omega}\times[\tilde{T}+2,T])}
+\|c\|_{C^{2,1}(\overline{\Omega}\times[\tilde{T}+2,T])}
+\|\mathbf{u}\|_{C^{2,1}(\overline{\Omega}\times[\tilde{T}+2,T])}
\leq c_3.
\end{align*}
Recalling that $T>\tilde{T}+2$ is arbitrarily chosen, we conclude the proof.
\end{proof}
\section*{Acknowledgements}
 Mengyao Ding is supported by the National Natural
 Science Foundation of China (12071009).

{
\setlength{\parskip}{0pt}
\setlength{\itemsep}{0pt plus 0.3ex}
\small

}
\end{document}